\documentclass[12pt]{amsart}

\pdfoutput=1

\usepackage{amsthm}
\usepackage{amsmath}
\usepackage{amstext}
\usepackage{amssymb}
\usepackage{bbm}
\usepackage{caption}
\usepackage{verbatim}
\usepackage{setspace}
\usepackage[dvips,letterpaper,margin= 1.3 in]{geometry}

\usepackage{graphicx}
\usepackage{hyperref}
\usepackage{pdflscape}
\usepackage{subcaption}

\usepackage{float}
\usepackage[utf8]{inputenc}
\usepackage[final]{pdfpages}

\newtheorem{theorem}{Theorem}[section]
\newtheorem{example}[theorem]{Example}
\newtheorem{lemma}[theorem]{Lemma}
\newtheorem{remark}[theorem]{Remark}
\newtheorem{proposition}[theorem]{Proposition}
\newtheorem{corollary}[theorem]{Corollary}
\newtheorem{conjecture}[theorem]{Conjecture}
\newtheorem{definition}[theorem]{Definition}

\newtheorem{question}[theorem]{Question}

\title{Discrete Equidecomposability and Ehrhart Theory of Polygons}

\begin{document}

\author{Paxton Turner}
\address{Department of Mathematics, Louisiana State University, Baton Rouge, Louisiana 70803}
\email{pturne7@tigers.lsu.edu}

\author{Yuhuai (Tony) Wu}
\address{Department of Mathematics, University of New Brunswick, Fredericton, New Brunswick E3B4N9}
\email{wu.yuhuai0206@unb.ca}

\thanks{This research was supported by an NSF grant to Brown's Summer@ICERM program.}

\begin{abstract}

Motivated by questions from Ehrhart theory, we present new results on discrete equidecomposability. Two rational polygons $P$ and $Q$ are said to be \emph{discretely equidecomposable} if there exists a piecewise affine-unimodular bijection (equivalently, a piecewise affine-linear bijection that preserves the integer lattice $\mathbb{Z} \times \mathbb{Z}$) from $P$ to $Q$. In this paper, we primarily study a particular version of this notion which we call \emph{rational finite discrete equidecomposability}. We construct triangles that are Ehrhart equivalent but not rationally finitely discretely equidecomposable, thus providing a partial negative answer to a question of Haase--McAllister on whether Ehrhart equivalence implies discrete equidecomposability. Surprisingly, if we delete an edge from each of these triangles, there exists an \emph{infinite} rational discrete equidecomposability relation between them. Our final section addresses the topic of infinite equidecomposability with concrete examples and a potential setting for further investigation of this phenomenon.

\end{abstract}

\maketitle

\tableofcontents

\begin{section}{Introduction}
\label{sec:introduction}

We review some facts from Ehrhart theory to motivate our central concept of study, discrete equidecomposability. Some basic notation is also developed in this section. The introduction closes with an outline of our results and the structure of this paper.

\begin{subsection}{Ehrhart Theory}
Ehrhart theory, developed by Eug\`ene Ehrhart in the paper \cite{ehrhart} is the study of enumerating integer lattice points in integral dilates of (not necessarily convex) polytopes. Intuitively, this can be viewed as studying the \emph{discrete area} of dilates of a polytope. Given $t \in \mathbb{N}$, the $t$'th dilate of $P$ is the set $tP = \{tp \, | \, p \in P \}$. Here $tp$ denotes scalar multiplication of the point $p$ by $t$. With this set-up, we can define the central object of Ehrhart theory, the function $\mathrm{ehr}_P(t)$ that counts the number of lattice points in $tP$:

\begin{equation}
\label{eqn:ehrhart}
\mathrm{ehr}_P(t) := |\{tP \cap \mathbb{Z} \times \mathbb{Z} \}|.
\end{equation}

We say that $P$ is an \emph{integral polytope}, or simply, $P$ is \emph{integral} if its vertices lie in the lattice $\mathbb{Z} \times \mathbb{Z}$. Similarly, a \emph{rational polytope} has all of its vertices given by points whose coordinates are rational. If $P$ is rational, the \emph{denominator} of $P$ is defined to be the least natural number $N$ such that $NP$ is an integral polytope. Equivalently, the denominator is the least common multiple of the set of denominators of the coordinates of the vertices of $P$.

A fundamental theorem due to Ehrhart states that $\mathrm{ehr}_P(t)$ has a particularly nice structure for rational polytopes. In this case, $\mathrm{ehr}_P(t)$ is a \emph{quasi-polynomial}.\footnote{The same cannot be said, however, for polygons with irrational vertices. See the analysis by Hardy and Littlewood of lattice points in a right-angled triangle \cite{hardy}.} A quasi-polynomial $q$ is a function of the following form. The $q_i$ below are all polynomials of the same degree known as the \emph{constituents} of $q$, and $d$ is a positive integer known as the \emph{period} of $q$.

\begin{displaymath}
   q(t) = \left\{
     \begin{array}{lr}
       q_1(t) & : t \equiv 1 \mod d \\
       q_2(t) & : t \equiv 2 \mod d \\
       \hspace{1.3 cm} \vdots \\
       q_j(t) & : t \equiv j \mod d \\
       \hspace{1.3 cm} \vdots \\
       q_d(t) & : t \equiv d \mod d \\
     \end{array}
   \right.
\end{displaymath}

This key theorem of Ehrhart is stated below.

\begin{theorem}[Ehrhart]
Let $P$ be a rational polytope of denominator $d$. Then $\mathrm{ehr}_P(t)$ is a quasi-polynomial of period $d$.
\end{theorem}

An exposition of the proof of this theorem is available in Ehrhart's original paper \cite{ehrhart} and also the textbook \cite{beck} by Beck and Robins.

We remark that the denominator of $P$ is not necessarily the \emph{smallest} period of $\mathrm{ehr}_P(t)$. In fact, there is a result due to McAllister and Woods showing that given $d$ and $k|d$, there exists a denominator $d$ rational polytope of arbitrary dimension such that $\mathrm{ehr}_P(t)$has period $k$  (see Theorem 2.2 of \cite{woods}). When $\mathrm{ehr}_P(t)$ has period $k$ strictly smaller than the denominator of $P$, it is said that $P$ has \emph{period collapse} $k$. Period collapse is still a mysterious phenomenon. For discrete geometric/combinatorial directions, see the papers (\cite{woods, mcallister, haase}) and for connections with representations of Lie algebras, see the papers (\cite{deloera, deloera1, derksen, kirillov}).

Given $S \subset \mathbb{R}^n$, let $\mathrm{Conv}(S)$ denote the convex hull of the set $S$. A motivating instance of period collapse for our purposes is given by the following example due to McAllister and Woods \cite{woods}. Let $T = \mathrm{Conv}((0,0), (1,\frac{2}{3}), (3,0))$. Then McAllister and Woods show that $T$ has denominator $3$, but $\mathrm{ehr}_T(t)$ has period $1$. In this sense, $T$ behaves like an integral polygon even though its vertices live in $\frac{1}{3} \mathbb{Z} \times \frac{1}{3} \mathbb{Z}$.\footnote{Polygons with period collapse $1$ are referred to as \emph{pseudo-integral polygons} or \emph{PIP}s in the paper \cite{mcallister} by McAllister.}

\begin{figure}[H]
    \centering
    \includegraphics[keepaspectratio=true, width=13cm]{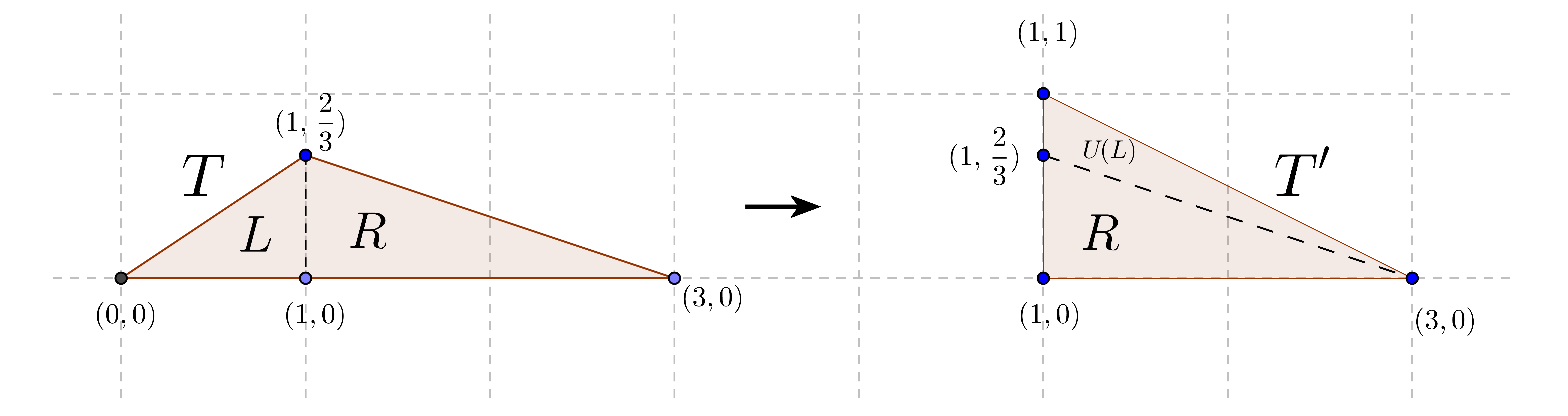}
    \caption{\small The cut-and-paste map from $T$ to $T'$. }
    \label{fig:equiexample}
\end{figure}

As such, we suspect that $T$ is, in some sense, \emph{equivalent} to an \emph{integral} polygon $T'$. For this example, this is true and the correct notion of equivalence is the concept of \emph{discrete equidecomposability} as defined in \cite{haase} and \cite{kantor1}.\footnote{However, there do exist denominator $d$ polygons having period collapse $1$ that are not discretely equidecomposable with any integral polygon. See Example 2.4 of \cite{haase}.} This means, informally speaking, that we can cut $T$ into pieces and map each piece by an affine-linear $\mathbb{Z} \times \mathbb{Z}$-preserving transformation so that the final result of this process is an integral polygon $T'$. The described procedure is demonstrated in Figure \ref{fig:equiexample}.

Let $E$ be the segment from $(1,0)$ to $(1,1/3)$, let $L = \mathrm{Conv}((0,0), (1,0), (1, 2/3)) \backslash E$, and let $R = T - L$. Then we map $T$ to $T'$ by leaving $R$ fixed and mapping $L$ by the action of the matrix $U = \begin{pmatrix} 1 & -2 \\ -1 & 1 \end{pmatrix}$.

The next section provides general definitions of discrete equidecomposability.

\end{subsection}

\begin{subsection}{Discrete Equidecomposability}

We restrict our attention hereafter to (not necessarily convex) polygons in $\mathbb{R}^2$, as all of our results concern this case. However, the concept of discrete equidecomposability and all the questions posed in this section make sense in higher dimensions.

The notion of discrete equidecomposability captures two sorts of symmetries: (1) affine translation and (2) lattice-preserving linear transformations. Note that both (1) and (2) preserve the number of lattice points in a region and, hence, Ehrhart quasi-polynomials. The \emph{affine unimodular group} $G := GL_2(\mathbb{Z}) \rtimes \mathbb{Z}^2$ with the following action on $\mathbb{R}^2$ captures both properties. \\

\begin{center}
$x \mapsto gx := Ux + v$

$x \in \mathbb{R}^2, \, g = U \rtimes v \in G = GL_2(\mathbb{Z}) \rtimes \mathbb{Z}^2$.
\end{center}

Note that $GL_2(\mathbb{Z})$ is precisely the set of integer $2 \times 2$ matrices with determinant $\pm 1$. Two regions $R_1$ and $R_2$ (usually polygons for our purposes) are said to be $G$\emph{-equivalent} if they are in the same $G$-orbit, that is, $G R_1 = G R_2$. Also, a $G$\emph{-map} is a transformation on $\mathbb{R}^2$ induced by an element $g \in G$. We slightly abuse notation and refer to both the element and the map induced by the element as $g$.

In the same manner of \cite{haase}, we define the notion of discrete equidecomposability in $\mathbb{R}^2$.\footnote{The source \cite{haase} defines this notion in $\mathbb{R}^n$ and labels it $GL_n(\mathbb{Z})$\emph{-equidecomposability}. See definition 3.1.}

\begin{definition}[discrete equidecomposability]
\label{def:equi}

Let $P, Q \subset \mathbb{R}^2$. Then $P$ and $Q$ are \emph{discretely equidecomposable} if there exist open simplices $T_1, \ldots, T_r$ and $G$-maps $g_1, \ldots, g_r$ such that

\begin{equation*}
P = \bigsqcup _{i= 1} ^r T_i \, \, \mathrm{and} \, \, Q = \bigsqcup _{i = 1} ^r g_i(T_i).
\end{equation*}

Here, $\bigsqcup$ indicates disjoint union.

\end{definition}

If $P$ is a polygon, we refer to the collection of open simplices $\{ T_1, \ldots, T_r \}$ as a \textit{simplicial decomposition} or \textit{triangulation}. It is helpful to observe how Figure \ref{fig:equiexample} is an example of this definition.

If $P$ and $Q$ are discretely equidecomposable, then there exists a map $\mathcal{F}$ which we call the \emph{equidecomposability relation} that restricts to the specified $G$-map on each open piece of $P$. Precisely, that is, $\mathcal{F}|_{T_i} = g_i$. The map $\mathcal{F}$ is thus a piecewise $G$-bijection. Observe from the definition that the map $\mathcal{F}$ preserves the Ehrhart quasi-polynomial; hence $P$ and $Q$ are Ehrhart-equivalent if they are discretely equidecomposable.

Our results only concern \emph{rational} discrete equidecomposability, the case of Definition \ref{def:equi} where all simplices $T_i$ are rational. In Section \ref{sec:infinite} we will also present an example of an equidecomposability relation with an infinite number of rational simplices, referring to such a map as an \emph{infinite} rational equidecomposability relation.

\end{subsection}

\begin{subsection}{Results}

The paper \cite{haase} presented an interesting question that motivated this research.

\begin{question}
\label{que:ehr}
Are Ehrhart-equivalent rational polytopes always discretely equidecomposable?\footnote{Labeled Question 4.1 in \cite{haase}.}
\end{question}


We prove that the answer to Question \ref{que:ehr} is ``no'' in the case of \textit{rational} discrete equidecomposability. In Section \ref{sec:weight}, we present two denominator $5$ minimal triangles (as defined in Section \ref{sec:minimal}) that are Ehrhart equivalent but not rationally discretely equidecomposable. We prove this by means of a weighting system on the edges of denominator $d$ polygons that serves as an invariant for rational discrete equidecomposability.

\begin{theorem}[Main Result 1]
If two polygons are Ehrhart equivalent, then it is not always the case that they are rationally discretely equidecomposable.
\end{theorem}

However, it is possible that Ehrhart equivalence implies discrete equidecomposability if we allow irrational simplices or an infinite number of simplices in our decompositions. We have the following result in the realm of infinite rational discrete equidecomposability.

\begin{theorem}[Main Result 2]
There exists an infinite family of pairs of triangles $\{S_i, T_i\}$ with the following properties:

\begin{enumerate}
\item $S_i$ and $T_i$ are Ehrhart equivalent but not rationally discrete equidecomposable.
\item If a particular closed edge is deleted from both $S_i$ and $T_i$, then the modified triangles are infinitely rationally discretely equidecomposable.
\end{enumerate}
\end{theorem}




This theorem motivates the following conjecture presented in Section \ref{sec:infinite}.

\begin{conjecture}
Ehrhart equivalence is a necessary and sufficient condition for (not necessarily finite or rational) discrete equidecomposability.
\end{conjecture}

\end{subsection}

\begin{subsection}{Outline}

\begin{remark}
For the rest of this paper, we abbreviate the phrase \emph{rational finite discrete equidecomposability} with \emph{equidecomposability}, as all of our results (except for the construction and questions in Section \ref{sec:infinite}) concern this situation. Furthermore, we emphasize that we are only working with polytopes in dimension $2$ (polygons).
\end{remark}

\begin{itemize}
\item Section \ref{sec:minimal} takes advantage of the fact that any triangulation of an integral polygon can be reduced to a triangulation consisting of triangles that intersect lattice points only at their vertices.\footnote{This fact does not generalize to higher dimensions, and is thus one of the main barriers preventing our methods from being extended.} Such triangles are known as \emph{denominator} $1$ \emph{minimal triangles}. Minimal triangles of denominator $d$, which are scaled-down versions of denominator $1$ minimal triangles, serve as the building blocks for all of our constructions. Here we classify them via an action of the dihedral group on $3$ elements and also in terms of a $G$-invariant weighting scheme.

\item In Section \ref{sec:weight}, we show that this weighting scheme is a necessary invariant for equidecomposability. We produce two denominator $5$ minimal triangles $T_{(1,2)}$ and $T_{(1,4)}$ that are Ehrhart equivalent but have different weights. Therefore, they are not equidecomposable.

\item In Section \ref{sec:infinite}, we delete an edge from $T_{(1,2)}$ and $T_{(1,4)}$ and then produce an infinite rational equidecomposability relation between them. This is surprising since, by the main result of Section \ref{sec:weight}, there does not exist a finite rational equidecomposability relation between $T_{(1,2)}$ and $T_{(1,4)}$. By the same construction, we can produce an infinite family of pairs of triangles that share the same Ehrhart quasi-polynomial but are not equidecomposable. We also explore further questions related to infinite equidecomposability, and give definitions for equidecomposability in a broader way.

\item Section \ref{sec:questions} closes with questions for further research.

\end{itemize}

\end{subsection}

\end{section}

\begin{section}{$d$-Minimal Triangles}
\label{sec:minimal}

\begin{subsection}{Definitions and Motivation}

\begin{definition}
Let $\mathcal{L}_d = \frac{1}{d} \mathbb{Z} \times \frac{1}{d} \mathbb{Z}$. A triangle $T$ is said to be $d$\emph{-minimal} if $T \cap \mathcal{L}_d$ consists precisely of the vertices of $T$.
\end{definition}

It is well-known that $d$-minimal triangles have area $\frac{1}{2d^2}$ and that any denominator $P$ polygon can be triangulated by $d$-minimal triangles. See Section 3 of \cite{hermann} for proofs of both of these results. As a result, any triangulation $\mathcal{T}$ of $P$ can be refined into a $d'$-minimal triangulation $\mathcal{T}'$ for some positive integer $d'$. It is also useful to note that $G$-maps send $d$-minimal triangles to $d$-minimal triangles.

Now, in light of Definition \ref{def:equi}, an equidecomposability relation $\mathcal{F}:P \to Q$ can be viewed as bijecting a simplicial decomposition (that is, a triangulation) $\mathcal{T}_1$ of $P$ to a simplicial decomposition (triangulation) $\mathcal{T}_2$ of some polygon $Q$. That is, to each open simplex (face) in $\mathcal{T}_1$, we assign a $G$-map such that the overall map is a bijection. In this case we write $\mathcal{F}: (P, \mathcal{T}_1) \to (Q, \mathcal{T}_2)$.

Define the denominator of a triangulation $\mathcal{T}$ to be the least integer $N$ such that for all faces $F \in \mathcal{T}$, the dilate $NF$ is an integral simplex. Thus, observe that if $P$ is a denominator $d$ polygon with a triangulation $\mathcal{T}_1$ of denominator $d'$, then $d'$ is divisible by $d$ because the vertices of $P$ are included in $\mathcal{T}_1$. If $P$ and $Q$ are denominator $d$ polygons and $\mathcal{F}: (P, \mathcal{T}_1) \to (Q, \mathcal{T}_2)$, then $\mathcal{T}_1$ and $\mathcal{T}_2$ have the same denominator. To see this, observe that the vertices of $\mathcal{T}_1$ are bijected to the vertices of $\mathcal{T}_2$ via $G$-maps, and $G$-maps preserve denominators.

\begin{remark}
\label{rmk:denominator}
If $P$ and $Q$ are denominator $d$ polygons and $\mathcal{F}: (P, \mathcal{T}_1) \to (Q, \mathcal{T}_2)$, then $\mathcal{T}_1$ and $\mathcal{T}_2$ have the same denominator, call it $d'$. In this case, we write $\mathcal{F}_{d'}: (P, \mathcal{T}_1) \to (Q, \mathcal{T}_2)$ and say that $\mathcal{F}$ has denominator $d'$. By the preceding discussion, we see that $d'$ is divisible by $d$.
\end{remark}

Without loss of generality, we can refine $\mathcal{T}_1$ to a minimal triangulation $\mathcal{T}'_1$ (in some denominator) and let $\mathcal{F}$ instead act on $\mathcal{T}'_1$. Pointwise, the definition of $\mathcal{F}$ has not changed, we are simply changing the triangulation upon which we view $\mathcal{F}$ as acting. Therefore, when searching for equidecomposability relations between $P$ and $Q$, it suffices to consider equidecomposability relations defined on minimal triangulations of $P$ in all possible denominators. This is a crucial observation summarized in the following remark.

\begin{remark}
\label{rmk:equi}
Any equidecomposability relation $\mathcal{F}:P \to Q$ can be viewed as fixing a minimal triangulation  $\mathcal{T}_1$ (in some denominator) of $P$ and assigning a $G$-map $g_F$ to each face $F$ (vertex, edge, or facet, respectively) of $\mathcal{T}_1$ such that $g(F)$ is a face (vertex, edge, or facet, respectively) of some minimal triangulation $\mathcal{T}_2$ of $Q$. Hence, when enumerating equidecomposability relations with domain $P$, it suffices to consider those $\mathcal{F}$ that assign $G$-maps to the faces of some minimal triangulation of $P$.

\end{remark}

Therefore, it makes sense to to study the $d$-minimal triangles in general, especially their $G$-orbits. In this section, we will first classify the $G$-orbits of $d$-minimal triangles according to an action of the dihedral group on $3$ elements. Then we will introduce a $G$-invariant weighting system on edges of minimal triangles (known as \emph{minimal edges} that is crucial to the main results in Sections \ref{sec:weight}. Finally, we show that the $G$-orbit of a minimal triangle is classified by the weights of its minimal edges. This last result provides an explicit parametrization of the $G$-equivalence classes of $d$-minimal triangles.

\end{subsection}

\begin{subsection}{Classifying $G$-orbits of Minimal Triangles According to an Action of the Dihedral Group $D_3$}

The following properties of $G$-maps are pivotal for all of the results of this section.

\begin{remark}
\label{rmk:Gmap}
Suppose we have a $G$-map $g: P \to Q$. Since $g$ is an invertible linear map, it is a homeomorphism. Therefore $g: \partial P \to \partial Q$ is also a homeomorphism. Moreover, if $P$ and $Q$ are polygons, linearity and invertibility guarantee that edges are sent to edges and vertices are sent to vertices.

\end{remark}

The next proposition shows that it suffices to consider right triangles occuring in the unit square $[0,1] \times [0,1]$.

\begin{proposition}
\label{prop:standard}
Let $T_1 = \mathrm{Conv} \left((0,0), (\frac{1}{d}, 0), (0, \frac{1}{d}) \right)$. Then any $d$-minimal triangle $T$ is $G$-equivalent to some translation $T_1'$ of $T_1$ with respect to  the lattice $\mathcal{L}_d$. Reducing the vertices of $T_1'$ modulo $\mathbb{Z}^2$, we can further choose $T_1'$ to lie in the unit square.
\end{proposition}

\begin{proof}

Let $v$ be a vertex of $T$. Then $T-v$ is a triangle lying at the origin. Since $T$ is $d$-minimal, $dT - dv$ is a unimodular triangle. Applying a change of basis, there exists $U \in GL_2(\mathbb{Z})$ such that $U(dT - dv) = dT_1$. By linearity, $ U(dT - dv) = d U(T- v)$, which implies $U(T - v) = T_1$. Therefore, $U(T - v) + Uv = UT$ is a translation of $T_1$ with respect to the lattice $\mathcal{L}_d$. Reducing modulo $\mathbb{Z}^2$ (translating by integer vectors), we see that for some $w \in \mathbb{Z}^2$, $T_1' := UT + w$ lies in the unit square. The requirements of the proposition are satisfied.

\end{proof}

Thus, the question of classifying $G$ orbits reduces to understanding intersections of G orbits with the set of translations of $T_1$ in the unit square. The next proposition classifies the types of transformations that send a triangle of the form $T_1 + v$ to some $T_1 + w$.

\begin{figure}[H]
    \centering
    \includegraphics[keepaspectratio=true, width=6cm]{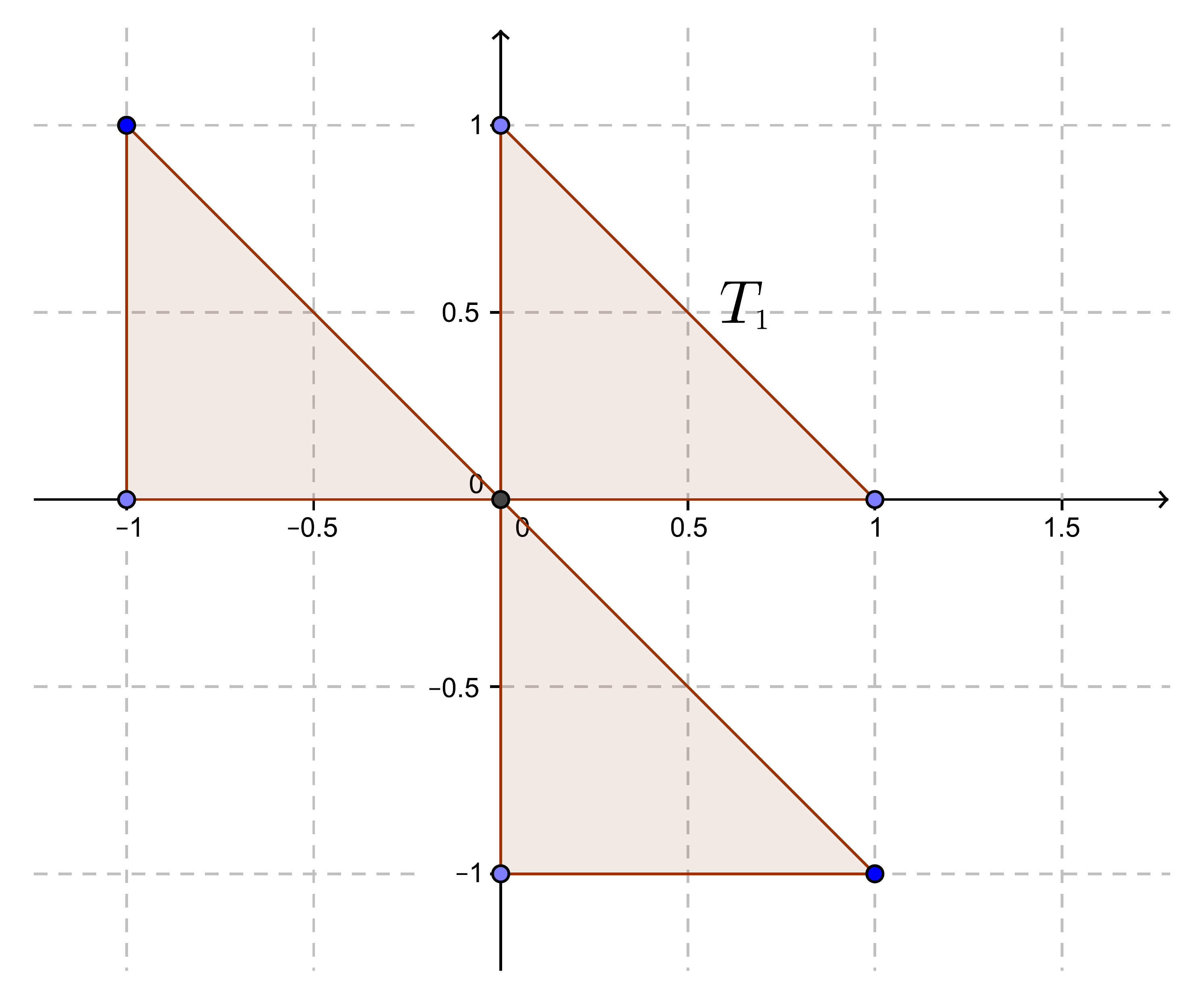}
    \caption{\small The possible images of $T_1$ under $U$. }
    \label{fig:U_classification}
\end{figure}

\begin{proposition}
\label{prop:dihedral}

Suppose $U(T_1 + v) + u = T_1 + w$ where $U \in GL_2(\mathbb{Z})$ and $u \in \mathbb{Z}^2$. Then $U$ belongs to the following set $D$ of matrices.

\begin{equation*}
\begin{split}
\left\{\begin{pmatrix} 1 & 0 \\ 0 & 1 \end{pmatrix}, \begin{pmatrix} 0 & 1 \\ 1 & 0 \end{pmatrix}, \begin{pmatrix} 0 & 1 \\ -1 & -1 \end{pmatrix}, \begin{pmatrix} -1 & -1 \\ 1 & 0 \end{pmatrix}, \begin{pmatrix} 1 & 0 \\ -1 & -1 \end{pmatrix}, \begin{pmatrix} -1 & -1 \\ 0 & 1 \end{pmatrix} \right\}
\end{split}
\end{equation*}
\end{proposition}

Note that $D$ defines a group isomorphic to the dihedral group $D_3 = \langle A, B | A^3 = B^2 = ABAB = I \rangle$. Simply set $A = \begin{pmatrix} -1 & -1 \\ 1 & 0 \end{pmatrix}$ and $B = \begin{pmatrix} 0 & 1 \\ 1 & 0 \end{pmatrix}$. This is the crucial observation required for proof of the main theorem later in this section.

\begin{proof}

If $U(T_1 + v) + u = T_1 + w$, then $UT_1 + Uv + u = T_1 + w$. Thus $UT_1 + Uv + u - w = T_1$. Note that invertible linear transformations preserve the vertices of a triangle by Remark \ref{rmk:Gmap}. Precisely, the vertices of the original triangle map to the vertices of its image. Hence, $Uv + u - w$ is a vertex of $T_1$. This implies that $UT_1$ is either $T_1$, $T_1 - (1,0)^{\mathrm{T}}$, or $T_1 - (0,1)^{\mathrm{T}}$. These triangles are shown in Figure \ref{fig:U_classification}.

Since $U$ is linear and sends vertices to vertices, we just need to compute the number of ways to send the ordered basis (and vertices of $T_1$) $\{(0,1), (1,0) \}$ to other non-zero vertices of the previous triangles listed. This amounts to computing six change of basis matrices, which are precisely given by the matrices in the set $D$.

\end{proof}

\vspace{-1.2 cm}

\begin{figure}[H]
    \centering
    \captionsetup{font=small,skip=0pt}
    \includegraphics[width=10cm, height = 9 cm]{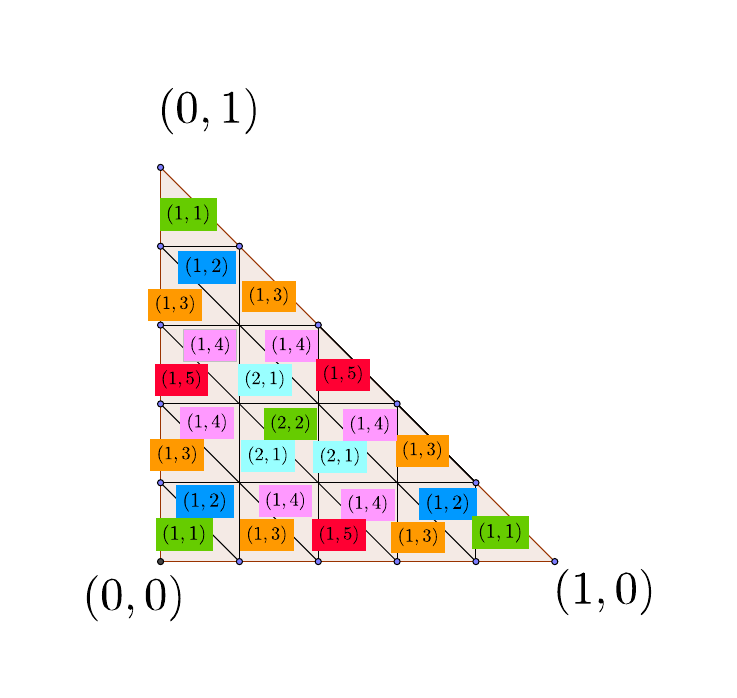}
    \caption{The triangulation $\mathcal{T}$ on $P$ in the case $d = 5$.}
    \label{fig:standard}
\end{figure}

The next proposition will accomplish our classification of $G$-orbits, as mentioned before, via a relationship with the dihedral group. Consider the half of the unit square given by $P =\mathrm{Conv}((0,0), (1,0), (0,1))$ and the $d$-minimal triangulation $\mathcal{T}$ of $P$ given by cutting grid-squares in half by lines of slope $-1$. An example is shown in Figure \ref{fig:standard}.

This triangulation involves translations of $T_1$ and translations of reflections of $T_1$ about the line $y = -x$. These reflected triangles are, however, equivalent under the lattice-preserving transformation given by the flip about the line $y = -x$ to translations of $T_1$ lying in the unit square. Hence, it sufficies to classify the $G$-orbits of the triangles lying in $P$. To simplify things, we can consider instead the quotient $\overline{P} := P + \mathbb{Z}^2 \subset \mathcal{L}_d / \mathbb{Z}^2$.

It is straightforward to check that we get an action of $D_3 \cong D$ on $\overline{P}$ and its triangulation by checking that each matrix in $D$ preserves the triangulation. That is, this action defines a permutation on $\mathcal{T}$, which is the key point. Let $\Phi$ be the bijection from $\overline{P}$ (and its constituent triangles) to a triangulated equilateral triangle $T$ as shown below for the case $n = 5$.

Now, we also have an action of $D$ on $T$ by letting $A$ act as a counterclockwise $60^{\circ}$ rotation and $B$ as a reflection about the angle bisector of the leftmost vertex, where $A$ and $B$ are defined in Proposition \ref{prop:dihedral}. We regard this action as a permutation on the set of triangles in the given triangulation of $T$. We claim that these two actions are compatible. This gives us an explicit understanding of the distribution of minimal triangles in $\overline{P}$.

\begin{proposition}
\label{prop:compatibility}
The actions of $D$ on the constituent triangles of $\overline{P}$ and on the constituent triangles of $T$ are compatible in the sense that given $\alpha \in D$,  $\alpha \Phi = \Phi \alpha$.
\end{proposition}

\begin{proof}

For both actions, the permutation on the set of constituent triangles is determined by the permutation of the vertices of the outer triangle (either $\overline{P}$ or $T$). Hence, it suffices to check that the actions on the vertices of each triangle are compatible. Let $v_1 = \Phi((0,0))$, $v_2 = \Phi((0,1))$, and $v_3 = \Phi((1,0))$ be the vertices of $T$. It is straightforward to check, using the definition of $\overline{P}$ as a quotient space, that $\Phi A(0,0) = \Phi (1,0)  = A \Phi (0,0)$, $\Phi A(1,0) = \Phi (0,1)  = A \Phi(1,0)$ and $\Phi A(0,1) = \Phi (0,0)  = A \Phi(0,1)$. The same remarks hold for the action of $B$. Since this is a group action, it suffices to check compatibility at the generators, so we conclude the proposition statement.

\end{proof}

\begin{figure}[H]
    \centering
    \includegraphics[keepaspectratio=true, width=15cm]{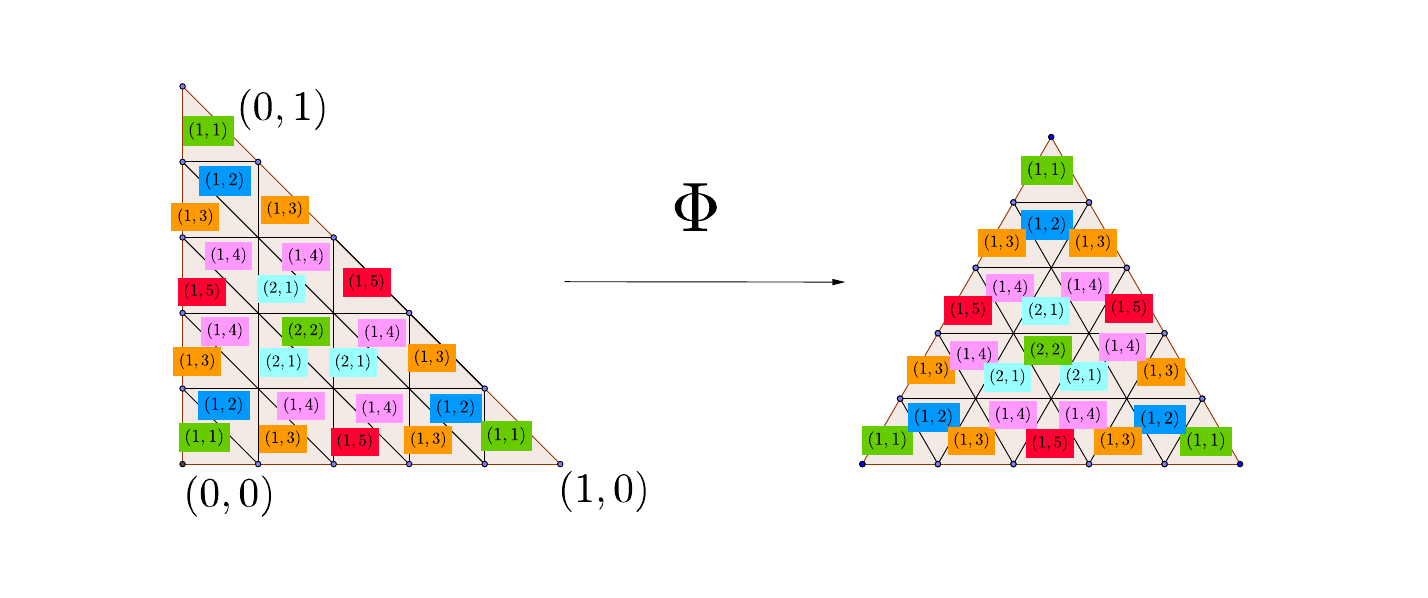}
    \caption{The map $\Phi$ in the case $d = 5$.}
    \label{fig:Phi_map}
\end{figure}

\end{subsection}

\begin{subsection}{A $G$-invariant Minimal Edge Weighting System}

Minimal segments are the $1$-dimensional counterparts of minimal triangles.

\begin{definition}[minimal edge]
\label{def:minimal_edge}

A line segment $E$ with endpoints in $\mathcal{L}_d$ is said to be a \emph{$d$-minimal segment} if $E \cap \mathcal{L}_d$ consists precisely of the endpoints of $E$.

\end{definition}

In particular, observe that the edges of $d$-minimal triangles are $d$-minimal segments. Our goal in the next two subsections is to develop a $G$-invariant weighting system on minimal edges that we will extend to a weighting system on minimal triangles. The existence of this invariant is the key to all of our main results.

Note that this weight is defined on \emph{oriented} minimal edges: minimal edges with an ordering/direction on its endpoints. For example, if $E$ has vertices $p$ and $q$, we can assign $E$ the orientation $p \to q$, in which case we can represent the oriented edge as $E_{p \to q}$.

\begin{definition}[weight of an edge]
\label{def:weight}

Let $E_{p \to q}$ be an oriented $d$-minimal edge from endpoint $p = (\frac{w}{d}, \frac{x}{d})$ to $q = (\frac{y}{d}, \frac{z}{d})$. Then define the weight $W(E_{p \to q})$ of $E_{p \to q}$ to be

\begin{equation*}
W(E_{p \to q}) = \det \left[ \begin{pmatrix} d & 0 \\ 0 & d \end{pmatrix} \begin{pmatrix} w/d & y/d \\ x/d & z/d \end{pmatrix} \right] = \det \begin{pmatrix} w & y \\ x & z \end{pmatrix} \mod d.
\end{equation*}

\end{definition}

As mentioned, $W$ is invariant (up to sign) under the action of $G$.

\begin{proposition}[$W$ is $G$-invariant]
\label{prop:weight_inv}

Let $E_{p \to q}$ be an oriented minimal edge and let $g \in G$. Then $W(E_{p \to q}) = \pm W(g(E)_{g(p) \to g(q)})$. Precisely, if $g$ is orientation preserving (i.e. $\det g=1$), $W(E_{p \to q}) = W(g(E)_{g(p) \to g(q)})$, and if $g$ is orientation reversing (i.e. $\det g=-1$), $W(E_{p \to q}) = -W(g(E)_{g(p) \to g(q)})$.

\end{proposition}

\begin{proof}
Let $g = U \rtimes v \in G$. Write

\begin{align*}
U = \begin{pmatrix} u_{11} & u_{12} \\ u_{21} & u_{22} \end{pmatrix}; \, v = \begin{pmatrix} v_1 \\ v_2 \end{pmatrix} \, \\
p = \begin{pmatrix} w/d \\ x/d \end{pmatrix}; \,
q = \begin{pmatrix} y/d \\ z/d \end{pmatrix}.
\end{align*}

Then $g(E) = \mathrm{Conv} \left(g(p), g(q) \right)$. That is, $g(p)$ and $g(q)$ are the endpoints of $g(E)$. Also, observe that $\det U = \pm 1$. Therefore, by Definition \ref{def:weight},

\begin{align*}
W(g(E)_{g(p) \to g(q)}) = \\
\det \begin{pmatrix} d & 0 \\ 0 & d \end{pmatrix}
\left[ \begin{pmatrix} u_{11} & u_{12} \\ u_{21} & u_{22} \end{pmatrix}\begin{pmatrix} w/d & x/d \\ y/d & z/d \end{pmatrix} + \begin{pmatrix} v_1 & v_1 \\ v_2 & v_2 \end{pmatrix} \right] \mod d = \\
\det \begin{pmatrix} u_{11} & u_{12} \\ u_{21} & u_{22} \end{pmatrix}
\begin{pmatrix} w & x \\ y & z \end{pmatrix} \mod d = \\
\det \begin{pmatrix} u_{11} & u_{12} \\ u_{21} & u_{22} \end{pmatrix}
\det \begin{pmatrix} w & x \\ y & z \end{pmatrix} \mod d = \\
\pm \det \begin{pmatrix} w & x \\ y & z \end{pmatrix} \mod d = \\
\pm W(E_{p \to q}).
\end{align*}

By analyzing the previous calculation (in particular lines 4 and 5), we also recover the last statement of the proposition.

\end{proof}

We state some nice geometric interpretations of the weight which will be critical to the proofs in Section \ref{subsec:classify}. Although our use of the determinant in this particular setting of discrete equidecomposability is new, both of the following geometric properties previously known from properties of determinants and the theory of lattices.




\begin{proposition}[area interpretation of the weight $W$]
\label{prop:area}
Let $E$ be a non-oriented $d$-minimal segment with endpoints $p = (w/d, x/d)$ and $q = (y/d, z/d)$. Construct the (perhaps degenerate) triangle $T$ with vertices at the origin, $p$, and $q$. Give $T$ the counterclockwise orientation (if $T$ degenerates to a segment, orient that segment in either direction). Suppose WLOG that this induces the orientation $p \to q$ on $E$. Let $Area(T)$ denote the relative area of $T$ in the lattice $\mathcal{L}_d$ where the fundamental parallelogram in $\mathcal{L}_d$ (here a square) is given area $1$. Then

\begin{equation*}
W(E_{p \to q}) = 2 \mathrm{Area}(T) \mod d.
\end{equation*}

\end{proposition}








We omit the proof of this fact; it simply relies on calculating area of a triangle via the cross-product formula. There are a few more notions we need to define before stating the second geometric property.

\begin{definition}
\label{def:counterclock}
Let $E$ be an oriented $d$-minimal edge with endpoints $p$ and $q$. If $p$ and $q$ do not lie on a line through the origin, the edge $E$ is said to be \emph{oriented counterclockwise (clockwise, respectively)} if it has the induced orientation by the counterclockwise (clockwise, respectively) oriented triangle $T$ with vertices at the origin, $p$, and $q$.
If $p$ and $q$ lie on a line that passes through the origin, then by convention, $E$ is said to be both \emph{oriented counterclockwise and clockwise}.
\end{definition}

\begin{figure}[H]
    \centering
    \includegraphics[keepaspectratio=true, width=11cm]{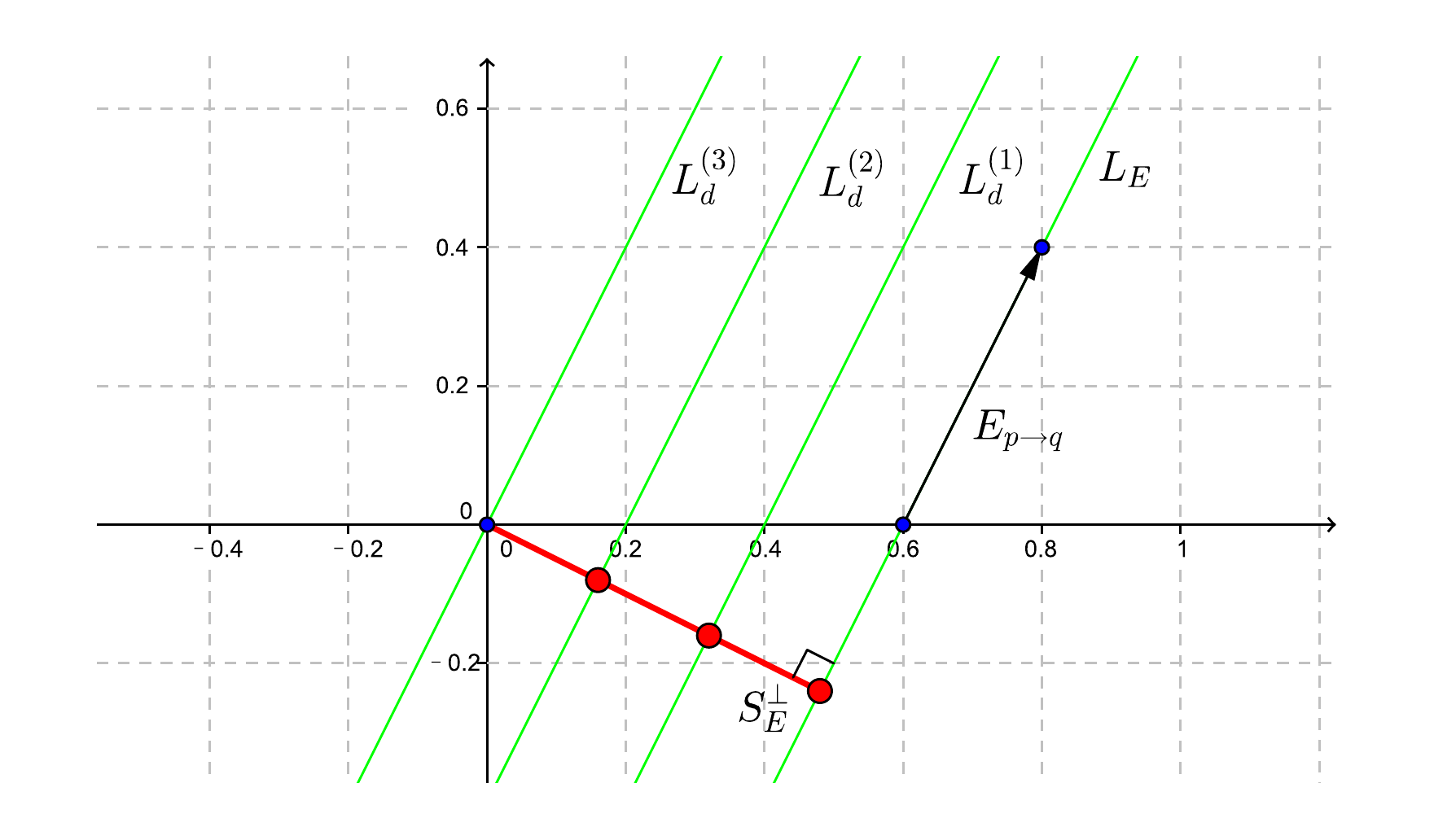}
    \caption{\small $\mathrm{dis}(E_{p \to q}) = 4-1=3$ in this case, where E is the edge with end points $p$, $(0.6,0)$ and $q$, (0.8,0.4) (both lie in $\mathcal{L}_5$). The $L_E^{(i)}$ indicate the lines between $L_E$ and the origin.}
    \label{fig:lattice_dist}
\end{figure}

Now we define $\mathrm{dis}(E)$, the \emph{lattice-distance} of a $d$-minimal segment $E$ from the origin. A line in $\mathbb{R}^2$ is said to be an $\mathcal{L}_d$-\emph{line} if its intersection with $\mathcal{L}_d$ is non-empty. Let $L_E$ be the line extending the segment $E$ and call $L_E^{\parallel}$ the set of all $\mathcal{L}_d$-lines parallel to $L_E$. Construct $S_E^{\perp}$, the (closed) line segment perpendicular to $L_E$ from the origin to $L_E$. Finally, we may define $\mathrm{dis}(E)$ formally as follows.

\begin{equation}
\mathrm{dis}(E) := \left| \{ L_E^{\parallel} \cap S_E^{\perp} \}\right| - 1
\label{eqn:lattice_distance}
\end{equation}

In words, $\mathrm{dis}(E)$ is the number of $\mathcal{L}_d$ lines parallel to $E$  between the origin and the line $L_E$ containing $E$, inclusive, minus one for the line through the origin.

Our weight invariant computes $\mathrm{dis}(E) \mod d$ up to sign. This is a well-known fact in the theory of lattices that we attribute to folklore.

\begin{proposition}[lattice-distance interpretation of the weight $W$]
\label{prop:lattice_distance}
Let $E_{p \to q}$ be a counterclockwise oriented minimal segment. Then

\begin{equation}
W(E_{p \to q}) = \mathrm{dis}(E_{p \to q}) \mod d.
\label{eqn:dist}
\end{equation}

\end{proposition}

\begin{proof}

Note from Proposition \ref{prop:area} that $W(E_{p \to q}) = 2 \mathrm{Area}(T) \mod d$, where $T$ is the triangle with vertices $p$, $q$, and the origin. Let $E$ be the base of the triangle $T$. Then the relative length of $E$ in $\mathcal{L}_d$ is $1$, since $E$ is minimal. Observe that the relative length $h$ of the height of triangle $T$ from base $E$ is given by the relative length of $S_E^{\perp}$ (see Figure \ref{fig:lattice_dist}). Thus, by definition of lattice distance, $h = \mathrm{dis}(E)$.

Therefore, we see $2 Area(T) = 2 (\frac{1}{2})(1) \mathrm{dis}(E) = \mathrm{dis}(E) \mod d$.

The proposition statement follows by Proposition \ref{prop:area} because $E_{p \to q}$ is oriented counterclockwise: $W(E_{p \to q}) = 2 Area(T) = \mathrm{dis}(E_{p \to q}) \mod d$.

\end{proof}

\end{subsection}

\begin{subsection}{Classifying $G$-orbits of Minimal Triangles via Weights}
\label{subsec:classify}

Using the geometric properties described in Propositions \ref{prop:dihedral} and \ref{prop:lattice_distance}, we can classify $d$-minimal triangles by the weights of their minimal edges. We begin with the definition of this weight.

\begin{definition}
\label{def:weight_triangle}
Let $T$ be a $d$-minimal triangle with vertices $p, q,$ and $r$. Orient $T$ counterclockwise, and suppose WLOG this orients the edges $E^1, E^2, E^3$ of $T$ as follows: $E^1_{p \to q}, E^2_{q \to r},$ and $E^3_{r \to p}$. Then we define the weight $W(T)$ of $T$ to be the (unordered) multiset as follows:

\begin{equation}
W(T) := \left\{ W(E^1_{p \to q}), W(E^2_{q \to r}), W(E^3_{r \to p}) \right\}.
\label{eqn:weight_triangles}
\end{equation}

\end{definition}

First, we observe that sum of weights of a minimal triangle is equal to $1 \mod d$.

\begin{proposition}
\label{prop:area2}
Suppose $W(T) = \{a, b, c\}$. Then $a + b + c \equiv 1 \mod d$.
\end{proposition}

\begin{proof}
Use the same setup as in Definition \ref{def:weight_triangle}. Let $T$ be a $d$-minimal triangle with vertices $p, q,$ and $r$. Orient $T$ counterclockwise, and suppose WLOG this orients the edges $E^1, E^2, E^3$ of $T$ as follows: $E^1_{p \to q}, E^2_{q \to r},$ and $E^3_{r \to p}$. Then

\begin{equation*}
\left\{ W(E^1_{p \to q}), W(E^2_{q \to r}), W(E^3_{r \to p}) \right\} = \{a, b, c\}.
\end{equation*}

Moreover, assume $p = (u/d, v/d)$, $q = (w/d, x/d)$, and $r = (y/d, z/d)$. Since $T$ is oriented counterclockwise, its relative area in $\mathcal{L}_d$ is given by the following determinant (this is a fact from high-school analytic geometry)

\begin{equation*}
Area(T) = \frac{1}{2} \det \begin{pmatrix} 1 & 1 & 1 \\ u & w & y \\ v & x & z \end{pmatrix}.
\end{equation*}

Now computing this determinant with cofactor expansion, we have

\begin{equation*}
2 Area(T) =  \det \begin{pmatrix} w & y \\ x & z \end{pmatrix} - \det \begin{pmatrix} u & y \\ v & z \end{pmatrix} + \det \begin{pmatrix} u & w \\ v & x \end{pmatrix}.
\end{equation*}

Take residues of the above modulo $d$ and recall that the relative area of a $d$-minimal triangle is $\frac{1}{2}$, we have

\begin{equation*}
\begin{split}
2 Area(T) \equiv 1 \equiv W(E_{q \to r}) - W(E_{p \to r}) + W(E_{p \to q}) \mod d \\
\equiv W(E_{q \to r}) + W(E_{r \to p}) + W(E_{p \to q}) \equiv a + b + c \mod d.
\end{split}
\end{equation*}

\end{proof}

Next, this weighting is strictly invariant (not just up to sign) under the action of $G$.

\begin{proposition}[W(T) is invariant under $G$.\footnote{In fact $W(T)$ is preserved by equidecomposability relations as well, but the proof of this is more difficult. See Theorem \ref{thm:invariant}.}]
\label{prop:weight_invariant}
Let $T$ be a minimal triangle, and $g$ a $G$-map. Then

\begin{equation*}
W(T) = W(g(T)).
\end{equation*}
\end{proposition}

\begin{proof}
This is a direct consequence of Proposition \ref{prop:weight_inv}. Suppose $T$ has vertices $p, q,$ and $r$, and that when oriented counterclockwise, its edges $E^1, E^2, E^3$ are oriented as follows: $E^1_{p \to q}, E^2_{q \to r}, E^3_{r \to p}$. Recall by Remark \ref{rmk:Gmap} that $G$-maps send facets to facets, edges to edges, and vertices to vertices.

If $g$ is orientation preserving, then $g(T)$ when oriented counterclockwise has its edges oriented as follows: $g(E^1)_{g(p) \to g(q)}, g(E^2)_{g(q) \to g(r)}, g(E^3)_{g(r) \to g(p)}$. Now use Proposition \ref{prop:weight_inv} and the assumption that $g$ is orientation preserving to conclude the statement of Proposition \ref{prop:weight_invariant}.

Likewise, if $g$ is orientation reversing, then $g(T)$ when oriented counterclockwise has its edges oriented in the reverse of the previous case. They would read as follows: $g(E^1)_{g(q) \to g(p)}, g(E^2)_{g(r) \to g(q)}, g(E^3)_{g(p) \to g(r)}$. Now use Proposition \ref{prop:weight_inv} and the assumption that $g$ is orientation reversing to conclude the statement of Proposition \ref{prop:weight_invariant}.
\end{proof}

Now we can show that the weight of a $d$-minimal triangle determines its $G$-orbit.

\begin{theorem}
\label{thm:weight_classification}
Two $d$-minimal triangles $S$ and $T$ are $G$-equivalent if and only if $W(S) = W(T)$.
\end{theorem}

\begin{proof}
We showed the left to right direction in the previous result, Proposition \ref{prop:weight_invariant}.

Now, suppose $W(S) = W(T)$. Use Proposition \ref{prop:standard} to map $S$ and $T$ to translations $S'$ and $T'$, respectively, of the triangle $T_1 = \mathrm{Conv} \left((0,0), (\frac{1}{d}, 0), (0, \frac{1}{d}) \right)$.

Let's temporarily order the sets $W(S')$ and $W(T')$, starting from the hypotenuse and reading off weights counterclockwise.

By Propositions \ref{prop:dihedral} and \ref{prop:compatibility}, we may act on $S'$ by a $G$-map $g$ such that (1) $g(S')$ is still a translation of $T_1$ and (2) the ordered weight of $g(S')$ is a permutation of the ordered weight of $S'$. In fact, all permutations of orderings are possible because the dihedral group on $3$ elements is precisely the symmetric group on $3$ elements. See Figure \ref{fig:perms} for a particular permutation.

\begin{figure}[H]
    \includegraphics[keepaspectratio=true, width=17 cm]{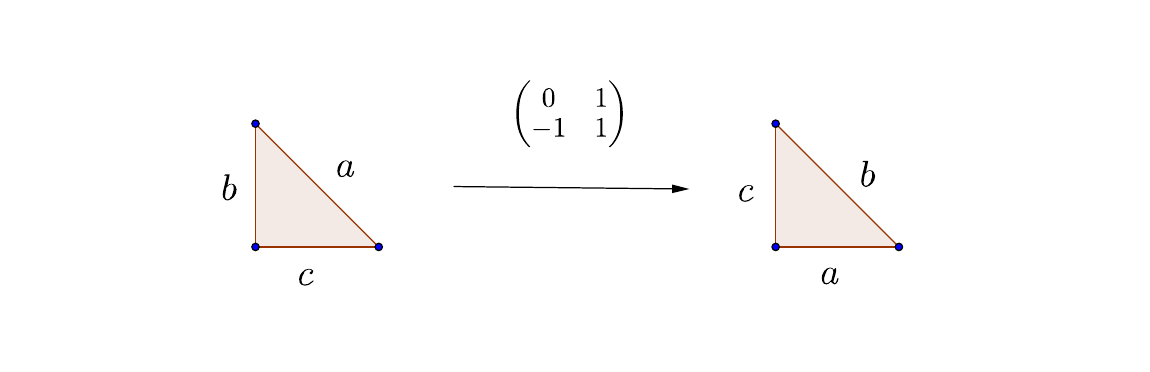}
    \caption{Each triangle in this figure is a translate of $T_1$. The matrix shown sends the ordered weight $\{ a, b, c\}$ to $\{b, a, c \}$. In general, each matrix from Proposition \ref{prop:dihedral} acts as one of the permutions in the symmetric group of three letters on $\{a, b, c\}$.}
    \label{fig:perms}
\end{figure}

Thus, we may choose a map $g$ so that the ordered weight of $T'$ and $S'' = g(S')$ agree. By Proposition \ref{prop:lattice_distance}, the lattice distance of the vertical (respectively, horizontal) edges of $T'$ and $S''$ agree modulo $d$. Therefore, the coordinates of the vertex of $T'$ opposite the hypotenuse must agree with the coordinates of the vertex of $S''$ opposite its hypotenuse modulo integer translation.

Since $S''$ and $T'$ have the same geometric form, we conclude that $S''$ is an integer translate of $T'$. Thus, $S$ and $T$ are $G$-equivalent.
\end{proof}

\end{subsection}

\end{section}

\begin{section}{Weight is an Invariant for Equidecomposability}
\label{sec:weight}

Our goal in this section is to generalize the weight $W$ to arbitrary rational polygons (not just $d$-minimal triangles) and show that it serves as an invariant for equidecomposability. That is, if $P$ and $Q$ are equidecomposable rational polygons, we will show $W(P) = W(Q)$. This enables us to provide a negative answer to Question \ref{que:ehr} of Haase--McAllister \cite{haase} in the case of finite rational discrete equidecomposability.

Explicitly, the $5$-minimal triangles labeled $(1,2)$ and $(1,4)$ on the bottom row of Figure \ref{fig:Phi_map} have the same Ehrhart quasi-polynomial (we show this with the aid of the computational software \texttt{LattE} \cite{LattE}) but have differing weights. Therefore, they cannot be finitely rationally discretely equidecomposable.

\begin{subsection}{Weight of a Rational Polygon}

We extend the notion of weight to arbitrary rational polygons below.

\begin{definition}[weight of a rational polygon]
\label{def:weight_poly}

Let $P$ be a counterclockwise oriented denominator $d$ polygon and $d'$ a positive integer divisible by $d$. Then we may uniquely regard the boundary of $P$ as a finite union $\sqcup E^i$ of oriented $d'$-minimal segments $\{ E^i \}$. We define the $d'$-weight $W_{d'}(P)$ of the polygon $P$ to be the following unordered multiset.

\begin{equation*}
W_{d'}(P) := \bigcup \left\{ W(E_i) \right\}.
\end{equation*}

\end{definition}

Observe that for a $d$-minimal triangle $T$, $W_d(T)$ agrees with $W(T)$ as described by Definition \ref{def:weight_triangle}. We will work an example for clarity.

\begin{example}
\label{exl:counterexample}
\normalfont

Let $T_{(1,2)} = \mathrm{Conv}((1/5,0), (0,1/5), (1/5,1/5))$ and \\ $T_{(1,4)} = \mathrm{Conv}((2/5,0),(1/5,1/5),(2/5,1/5))$. The triangles $T_{(1,2)}$ and $T_{(1,4)}$ are the denominator $5$ triangles labeled $(1,2)$ and $(1,4)$, respectively, on the bottom row of the right hand side of Figure \ref{fig:Phi_map}.

We compute the edge-weights in the multiset $W_5(T_{(1,2)})$ below. To do so, we orient $T_{(1,2)}$ counterclockwise. Also, in such computations modulo $d$, for our purposes, it is convenient to select our set of residues to be centered around $0$. For example, if $d = 5$, we choose our residues from the set $\{-2, -1, 0, 1, 2\}$.

\begin{align*}
W_5(E^1_{(0,1/5) \to (1/5,0)}) = \det \begin{pmatrix} 0 & 1 \\ 1 & 0 \end{pmatrix} \mod d = -1 \\
W_5(E^2_{(1/5, 1/5) \to (0, 1/5)}) = \det \begin{pmatrix} 1 & 0 \\ 1 & 1 \end{pmatrix} \mod d = 1 \\
W_5(E^3_{(1/5,0) \to (1/5,1/5)}) = \det \begin{pmatrix} 1 & 1 \\ 0 & 1  \end{pmatrix} \mod d = 1 \\
\end{align*}

So $W_5(T_{(1,2)}) = \{1, 1, -1 \}$. In the same fashion, we can compute $W_5(T_{(1,4)}) = \{2, -2, 1\}$. Observe that $W_5(T_{(1,2)}) \neq W_5(T_{(1,4)})$.

\end{example}

\end{subsection}

\begin{subsection}{$W_d$ is an Invariant for Equidecomposability}

Recall from Remark \ref{rmk:equi} and the discussion preceding that an equidecomposability relation $\mathcal{F}: P \to Q$ between denominator $d$ polygons may be regarded as an assignment of $G$-maps to a $d'$-minimal triangulation (precisely, an open simplicial decomposition consisting of $d'$ minimal edges and facets) $\mathcal{T}_1$ of $P$.\footnote{This is the motivation for us adding an extra parameter $d'$ to the weight $W_{d'}(P)$ of a rational polygon $P$.} Since $\mathcal{F}$ is a bijection and $G$-maps preserve the lattice $\mathcal{L}_{d'}$, $\mathcal{T}_1$ is sent to a $d'$-minimal triangulation $\mathcal{T}_2$ of $Q$. From this point on, we will write $\mathcal{F}_{d'}: (P, \mathcal{T}_1) \to (Q, \mathcal{T}_2)$ to indicate the underlying triangulations and their denominator. This convention is summarized in the following remark.

\begin{remark}
\label{rmk:equi_min}
Given an equidecomposability relation $\mathcal{F}:P \to Q$ and $d'$-minimal triangulations $\mathcal{T}_1$ and $\mathcal{T}_2$ of $P$ and $Q$, respectively, we write $\mathcal{F}_{d'}: (P, \mathcal{T}_1) \to (Q, \mathcal{T}_2)$ if $\mathcal{F}|_{F}$ is a $G$-map for all faces $F \in \mathcal{T}_1$ and $\mathcal{F}(F)$ is a face of $\mathcal{T}_2$. Moreover, we say $\mathcal{F}$ has denominator $d'$.
\end{remark}

Since weights of facets (that is, $d'$-minimal triangles) in $\mathcal{T}_1$ are preserved by $G$-maps (see Proposition \ref{prop:weight_inv}), we see the multiset of weights of facets comprising $\mathcal{T}_1$ must be in bijection with the multiset of weights of facets comprising $\mathcal{T}_2$. Concretely, if there is a triangle with weight $\omega$ in the triangulation $\mathcal{T}_1$, there must be a triangle of weight $\omega$ in $\mathcal{T}_2$ and vice versa.

\begin{figure}[H]
    \centering
    \includegraphics[keepaspectratio=true, width=8cm]{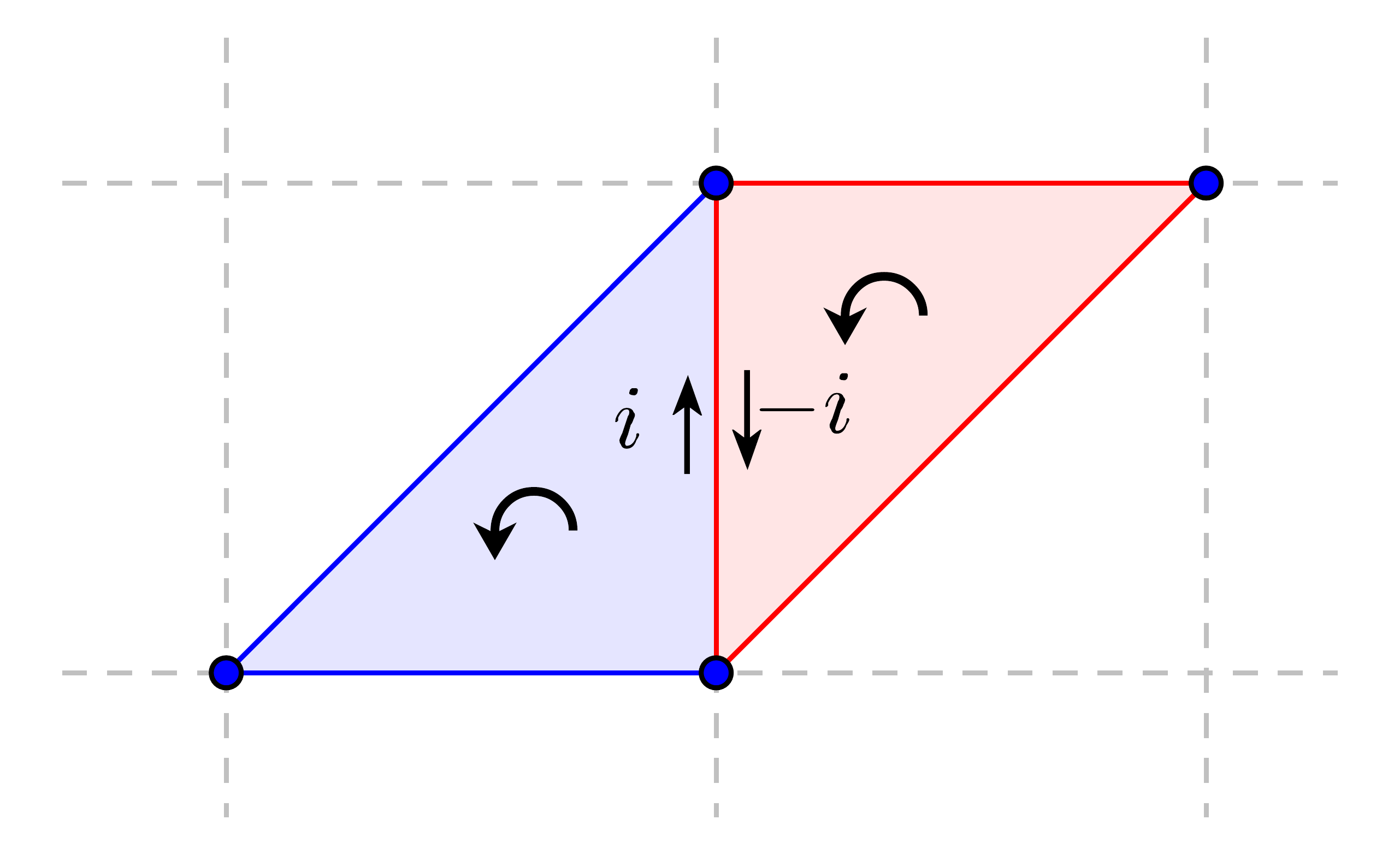}
    \caption{\small Two-sided edge-weighting system.}
    \label{fig:two_sided}
\end{figure}

The edges in $\mathcal{T}_1$ have a weighting (up to sign) induced by the weights of the facets. An edge $E$ in the interior of $\mathcal{T}_1$ is bordered by two facets $F_1$ and $F_2$ of $\mathcal{T}_1$. If the facet $F_1$ induces the weight $i$ on $E$, then $F_2$ induces the weight $-i$ on $E$, since both $F_1$ and $F_2$ are oriented counterclockwise. Hence, the induced weighting on the edges of $\mathcal{T}_1$ by the facets of $\mathcal{T}_1$ is properly regarded as a \emph{two-sided} edge-weighting system on interior edges.\footnote{This two-sided edge-weighting system on a triangulation is reminiscent, although not strictly speaking an example, of the \textit{current graphs} studied by Alpert and Gross \cite{gross, gross1}, Youngs \cite{youngs}, and Gustin \cite{gustin} in the setting of the Heawood map-coloring problem and related questions from topological graph theory. It is curious if there is a legitimate connection between our work and theirs, or if the similarity is only superficial.} See Figure \ref{fig:two_sided} for illustration. However, note that boundary edges in $\mathcal{T}_1$ only border a single facet in $\mathcal{T}_1$, and hence have a well-defined counterclockwise orientation and weight induced by the facet weights. In particular, the multiset of weights of boundary edges induced by the neighboring facets agrees with $W_{d'}(P)$.

To proceed further, we need the following definition.

\begin{definition}[$\pm i$ $d$-minimal edges]
\label{def:pmi_edges}
An oriented $d$-minimal edge $E$ is said to be a $\pm i$ $d$-minimal edge (or simply a $\pm i$ edge when the denominator is clear) if $W(E) = i$ or $W(E) = -i$.
\end{definition}

\begin{remark}
Note that the boundary edges of $\mathcal{T}_1$ with the unique weight of either $i$ or $-i$ induced by the orientation of $P$ are still considered to be $\pm i$ edges. Hence, the interior edges of $\mathcal{T}_1$ with a two-sided weight of $i$ and $-i$ as well as the boundary edges of $\mathcal{T}_1$ with a unique weight of either $i$ or $-i$ induced by the counterclockwise orientation on $P$ are all considered to be $\pm i$ $d'$-minimal edges according to Definition \ref{def:pmi_edges}.
\end{remark}

If $\mathcal{F}_{d'}: (P, \mathcal{T}_1) \to (Q, \mathcal{T}_2)$ is an equidecomposability relation of denominator $d'$ between denominator $d$ polygons $P$ and $Q$, we see by Proposition \ref{prop:weight_inv} the number of $\pm i$ $d'$-minimal edges in $\mathcal{T}_1$ is the same as the number of $\pm i$ weighted $d'$-minimal edges in $\mathcal{T}_2$. This is a $1$-dimensional analogue of the last statement of the second paragraph in this section. The preceding discussion is summarized by the following remark.

\begin{remark}
Suppose $\mathcal{F}_{d'}: (P, \mathcal{T}_1), (Q, \mathcal{T}_2)$ is an equidecomposability relation. Then

\begin{enumerate}
\label{rmk:preservation}
\item The multiset $\{W_{d'}(F)| \, F \, \mathrm{a} \, \mathrm{facet} \, \mathrm{in} \, \mathcal{T}_1 \}$ is in bijection with the multiset \\ $\{W_{d'}(F)| \, F \, \mathrm{a} \, \mathrm{facet} \, \mathrm{in} \, \mathcal{T}_2 \}$.
\item The set $\{E| \, E \, \mathrm{an} \, \mathrm{edge} \, \mathrm{in} \, \mathcal{T}_1, W(E) = \pm i \}$ is in bijection with the set \\ $\{E| \, E \, \mathrm{an} \, \mathrm{edge} \, \mathrm{in} \, \mathcal{T}_2, W(E) = \pm i \}$ for all residues $i$ modulo $d'$.
\end{enumerate}

\end{remark}

This motivates our next two definitions, the signed $d'$-weight $SW_{d'}(P)$ and unsigned $d'$-weight $UW_{d'}(P)$ of a rational denominator $d$ polygon $P$, where $d|d'$. These objects provide an indirect method of computing $W_{d'}(P)$ by playing off of cancellations and symmetries induced by the previously described two-sided weighting system on edges in a $d'$-minimal triangulation $\mathcal{T}_1$ of $P$. We can show with simple combinatorial arguments that $SW_{d'}(P)$ and $UW_{d'}(P)$ are invariant under denominator $d'$ equidecomposability relations, which implies, as we will show, that $W_{d'}(P)$ is also invariant under denominator $d'$ equidecomposability relations.

\begin{remark}
In all statements and definitions that follow in this section, $P$ and $Q$ denote denominator $d$ rational polygons, $d'$ is a positive integer divisible by $d$, and $\mathcal{F}_{d'}$ is an equidecomposability relation from $(P, \mathcal{T}_1)$ to $(Q, \mathcal{T}_2)$.
\end{remark}

\begin{definition}[signed $d'$-weight $SW_{d'}$]
\label{def:charge}
Fix a residue $i \mod d$ and let $\mathbbm{1}_i$ denote the indicator function of $i$ on the multiset $W_{d'}(P)$. Then $SW_{d'}(P)$ is a vector indexed by $\mathbb{Z}/d'\mathbb{Z}$ as follows.

\begin{equation*}
(SW_{d'}(P))_i = \sum _{j \in W_{d'}(P)} \mathbbm{1}_i(j) \, - \sum _{j \in W_{d'}(P)} \mathbbm{1}_{-i}(j)
\end{equation*}.

\end{definition}

\begin{example}
\normalfont
\label{exl:charge}
Recall triangles $T_{(1,2)}$ and $T_{(1,4)}$ from Example \ref{exl:counterexample}. Let's compute $SW_{5}$ of these two triangles. Recall that $W_5(T_{(1,2)}) = \{1, 1, -1 \}$ and $W_5(T_{(1,4)}) = (2, -2, 1)$. Let's represent $SW_{5}$ by a five-entry vector, starting with the $i = -2$ index and ending at the $i = 2$ index. Note that this accounts for all residues modulo $5$.

Then $SW_{5}(T_{(1,2)}) = SW_5(T_{(1,4)}) = \{0, -1, 0, 1, 0 \}$. As an example, let's show $SW_5(T_{(1,4)})_1 = 1$. By definition

\begin{equation*}
SW_5(T_{(1,4)})_1 = \sum _{j \in W_{5}(T_{(1,4)})} \mathbbm{1}_1(j) \, - \sum _{j \in W_{5}(T_{(1,4)})} \mathbbm{1}_{-1}(j) = 1 - 0 = 1.
\end{equation*}

\end{example}

\begin{definition}[unsigned $d'$-weight $UW_{d'}$]
\label{def:current}
Fix a residue $i \mod d$ and let $\mathbbm{1}_i$ denote the indicator function of $i$ on the multiset $W_{d'}(P)$. Then $UW_{d'}(P)$ is a vector indexed by $\mathbb{Z}/d'\mathbb{Z}$ as follows.

\begin{equation*}
(UW_{d'}(P))_i = \sum _{j \in W_{d'}(P)} \mathbbm{1}_i(j) \, + \sum _{j \in W_{d'}(P)} \mathbbm{1}_{-i}(j)
\end{equation*}.
\end{definition}

In other words, $(UW_{d'}(P))_i$ is the total number of edges in $W_{d'}(P)$ with weight $\pm i$.

\begin{example}
\label{exl:current}
\normalfont

Let's compute $UW_5$ of the triangles $T_{(1,2)}$ and $T_{(1,4)}$, recalling again that $W_5(T_{(1,2)}) = \{1, 1, -1 \}$ and $W_5(T_{(1,4)}) = \{2, -2, 1\}$. As in Example \ref{exl:charge}, let's index the 5-entry vector $UW_5$ so that the entries run starting from the index $i = -2$ and ending at the index $i = 2$. We see that $UW_5(T_{(1,2)}) = \{0,3 ,0,3 ,0 \}$ but $UW_5(T_{(1,4)}) = \{2, 1, 0, 1, 2 \}$. To be clear, we compute $UW_{5}(T_{(1,2)})_{-1}$ and $UW_5(T_{(1,4)})_{2}$ using the definition.

\begin{align*}
UW_{5}(T_{(1,2)})_{-1} = \sum _{j \in W_{d'}(T_{(1,2)})} \mathbbm{1}_{-1}(j) \, + \sum _{j \in W_{d'}(T_{(1,2)})} \mathbbm{1}_{1}(j) = 1 + 2 = 3 \\
UW_{5}(T_{(1,4)})_{2} = \sum _{j \in W_{d'}(T_{(1,4)})} \mathbbm{1}_{2}(j) \, + \sum _{j \in W_{d'}(T_{(1,4)})} \mathbbm{1}_{-2}(j) = 1 + 1 = 2.
\end{align*}

\end{example}

Now we prove the invariance under equidecomposability of the signed and unsigned weight.

One interpretation of the proof of Lemma \ref{lem:charge} is that $SW_{d'}$ is an \emph{additive valuation} (see \cite{mcmullen}).

\begin{lemma}[signed $d'$-weight invariance]
\label{lem:charge}
Let $\mathcal{F}_{d'}: (P, \mathcal{T}_1) \to (Q, \mathcal{T}_2)$. Then $SW_{d'}(P) = SW_{d'}(Q)$.
\end{lemma}

\begin{proof}

The key observation is that

\begin{equation}
\label{eqn:cancel}
\sum _{F \, \mathrm{a} \, \mathrm{facet} \, \mathrm{in} \, \mathcal{T}_1} SW_{d'}(F) = SW_{d'}(P)
\end{equation}
where we sum up the vectors $SW_{d'}(F)$ componentwise. Equation \ref{eqn:cancel} is justified as follows.

Let $i \mod d$ be a residue modulo $d$. We show

\begin{equation*}
\sum _{F \, \mathrm{a} \, \mathrm{facet} \, \mathrm{in} \, \mathcal{T}_1} SW_{d'}(F)_i = SW_{d'}(P)_i.
\end{equation*}

The LHS adds $1$ for every weight $i$ edge among the facets in $\mathcal{T}_1$ and adds $-1$ for every weight $-i$ edge among the facets in $\mathcal{T}_2$. Formally, we have:

\begin{equation}
\label{eqn:compute}
\begin{split}
\sum _{F \, \mathrm{a} \, \mathrm{facet} \, \mathrm{in} \, \mathcal{T}_1} SW_{d'}(F)_i =
\sum _{F \, \mathrm{a} \, \mathrm{facet} \, \mathrm{in} \, \mathcal{T}_1} \, \sum _{E \, \mathrm{an} \, \mathrm{edge} \, \mathrm{of} \, F} \mathbbm{1}_i(W(E)) - \mathbbm{1}_{-i}(W(E)) \\ = \sum _{F \, \mathrm{a} \, \mathrm{facet} \, \mathrm{in} \, \mathcal{T}_1} \, \sum _{E \, \mathrm{an} \, \mathrm{edge} \, \mathrm{of} \, F} \mathbbm{1}_i(W(E)) \, - \sum _{F \, \mathrm{a} \, \mathrm{facet} \, \mathrm{in} \, \mathcal{T}_1} \, \sum _{E \, \mathrm{an} \, \mathrm{edge} \, \mathrm{of} \, F} \mathbbm{1}_{-i}(W(E))
\end{split}
\end{equation}

Recall from Figure \ref{fig:two_sided} and the neighboring discussion that interior weight $i$ edges would appear once in the first summand of the RHS of Equation \ref{eqn:compute} with sign $+1$ and once in the second summand of the RHS of Equation \ref{eqn:compute} with sign $-1$. Therefore, the contribution of any interior edge to the sum in Equation \ref{eqn:compute} is $0$. Only the boundary edges with weight $i$ (when given the orientation induced by the counterclockwise orientation on $P$) need to be taken into account. Thus,

\begin{equation*}
\sum _{F \, \mathrm{a} \, \mathrm{facet} \, \mathrm{in} \, \mathcal{T}_1} SW_{d'}(F)_i = \sum _{\substack{E \in \mathcal{T}_1 \\ E \subset \partial P}} \mathbbm{1}_i(W(E)) - \mathbbm{1}_{-i}(W(E)) = SW_{d'}(P)_i,
\end{equation*}
where the last equality follows from unraveling Definitions \ref{def:charge} and \ref{def:weight_poly}. Thus, Equation \ref{eqn:cancel} holds.

Finally, the LHS of Equation \ref{eqn:cancel} is invariant under the equidecomposability relation $\mathcal{F}$, because $\mathcal{F}$ restricts to a $G$-map on facets of $\mathcal{T}_1$, and the weights of these facets are invariant under $G$-maps (see Proposition \ref{prop:weight_invariant} and \ref{rmk:equi_min}). That is,

\begin{equation*}
\begin{split}
\sum _{F \, \mathrm{a} \, \mathrm{facet} \, \mathrm{in} \, \mathcal{T}_1} SW_{d'}(F) = \sum _{\substack{\mathcal{F}(F) \, \mathrm{s.t.} \\ \, F \, \mathrm{a} \, \mathrm{facet} \, \mathrm{in} \, \mathcal{T}_1}} SW_{d'}(\mathcal{F}(F)) \\
= \sum  _{F' \, \mathrm{a} \, \mathrm{facet} \, \mathrm{in} \, \mathcal{T}_2} SW_{d'}(F') = SW_{d'}(Q).
\end{split}
\end{equation*}

The proof is complete.

\end{proof}

\begin{lemma}[unsigned $d'$-weight invariance]
\label{lem:current}
Let $\mathcal{F}_{d'}: (P, \mathcal{T}_1) \to (Q, \mathcal{T}_2)$. Then $UW_{d'}(P) = UW_{d'}(Q)$.
\end{lemma}

\begin{proof}

We introduce some notation to streamline this proof. Let $\mathbbm{1}_{\pm i} = \mathbbm{1}_{i} + \mathbbm{1}_{-i}$, the indicator function for $i$ or $-i$. Let

\begin{equation}
\label{eqn:pmi_triangle}
\begin{split}
\Delta_n^{\pm i} (\mathcal{T}_1) = \left\{ F| \, F \, \mathrm{is} \, \mathrm{a} \, \mathrm{facet} \, \mathrm{in} \mathcal{T}_1 \, \mathrm{and} \, \sum _{E \, \mathrm{an} \, \mathrm{edge} \, \mathrm{in} \, F} \mathbbm{1}_{\pm i}(W(E)) = n  \right\}.
\end{split}
\end{equation}

That is, $\Delta_n^{\pm i} (\mathcal{T}_1)$ is the set of facets in $\mathcal{T}_1$ having precisely $n$ edges of weight $\pm i$.

When the underlying triangulation is clear, we omit the argument $\mathcal{T}_1$. For the next part of this proof until Equation \ref{eqn:unsigned_count}, we will take the underlying triangulation $\mathcal{T}_1$ as implicit and simply write $\Delta_n^{\pm i}$ to represent $\Delta_n^{\pm i} (\mathcal{T}_1)$.

We make the following claim:

\begin{equation}
\label{eqn:current}
\begin{split}
\sum _{E \, \mathrm{an} \, \mathrm{edge} \, \mathrm{in} \, \mathcal{T}_1} \mathbbm{1}_{\pm i}(W(E)) = \frac{1}{2}\left(|\Delta_1^{\pm i}| + 2|\Delta_2^{\pm i}| + 3|\Delta_3^{\pm i}|\right) \\
+ \frac{1}{2} \sum _{\substack{E \, \mathrm{an} \, \mathrm{edge} \, \mathrm{in} \, \mathcal{T}_1 \\ E \subset \partial P}} \mathbbm{1}_{\pm i}(W(E)).
\end{split}
\end{equation}

To see this, note that the LHS of Equation \ref{eqn:current} counts the total number of edges in $\mathcal{T}_1$ with weight $\pm i$. Now let's understand the RHS.

Let $E$ be an interior edge of $\mathcal{T}_1$ with weight $\pm i$. Then $E$ is bordered by two facets $F_1$ and $F_2$ in $\mathcal{T}_1$. Both of these triangles are counted precisely once  by the term $|\Delta_1^{\pm i}| + 2|\Delta_2^{\pm i}| + 3|\Delta_3^{\pm i}|$. Therefore, $E$ is counted precisely once by the RHS of Equation \ref{eqn:current} (note the normalizing factor $\frac{1}{2}$), since $E$ is not a boundary edge.

Similarly, if $E$ is a boundary edge of $\mathcal{T}_1$, then $E$ only borders one facet in $\mathcal{T}_1$. Therefore, $E$ is counted once by the expression $|\Delta_1^{\pm i}| + 2|\Delta_2^{\pm i}| + 3|\Delta_3^{\pm i}|$. Furthermore, $E$ is counted precisely once by the term

\begin{equation*}
\sum _{\substack{E \, \mathrm{an} \, \mathrm{edge} \, \mathrm{in} \, \mathcal{T}_1 \\ E \subset \partial P}} \mathbbm{1}_{\pm i}(W(E))
\end{equation*}.

Since we normalize by $\frac{1}{2}$, $E$ is counted precisely once on the RHS of Equation \ref{eqn:current}.

We make the following observation using Definition \ref{def:current} and Equation \ref{eqn:current}.

\begin{equation}
\label{eqn:unsigned_count}
\begin{split}
UW_{d'}(P)_i = \sum _{\substack{E \, \mathrm{an} \, \mathrm{edge} \, \mathrm{in} \, \mathcal{T}_1 \\ E \subset \partial P}} \mathbbm{1}_{\pm i}(W(E)) = 2 \left( \sum _{E \, \mathrm{an} \, \mathrm{edge} \, \mathrm{in} \, \mathcal{T}_1} \mathbbm{1}_{\pm i}(W(E)) \right) \\
 - \left( |\Delta_1^{\pm i} (\mathcal{T}_1)| + 2|\Delta_2^{\pm i}(\mathcal{T}_1)| + 3|\Delta_3^{\pm i}(\mathcal{T}_1)| \right)
\end{split}
\end{equation}

We observe using Remark \ref{rmk:preservation} that the total number of $\pm i$ weighted edges in $\mathcal{T}_1$ is preserved by $\mathcal{F}$. That is, $\mathcal{T}_2$ has the same amount of $\pm i$ weighted edges as $\mathcal{T}_1$. Also by Remark \ref{rmk:preservation}, $|\Delta_n^{\pm i}(\mathcal{T}_1)| = |\Delta_n^{\pm i}(\mathcal{T}_2)|$. Therefore, the RHS of Equation \ref{eqn:unsigned_count} is preserved by $\mathcal{F}$, which implies $UW_{d'}(P)$ is preserved by $\mathcal{F}$, as desired.

\end{proof}

\begin{lemma}
\label{lem:weight_agree}
The weight $W_{d'}(P)$ is uniquely determined by the signed weight $SW_{d'}(P)$ and unsigned weight $UW_{d'}(P)$.
\end{lemma}

\begin{proof}

Let $n_i$ denote the number of times the residue $i$ occurs in the multiset $W_{d'}(P)$. From Definitions \ref{def:charge} and \ref{def:current} we see that $n_i - n_{-i} = SW_{d'}(P)_i$ and $n_i + n_{-i} = UW_{d'}(P)_i$. Therefore, $n_{i} = \frac{1}{2} (SW_{d'}(P) + UW_{d'}(P))$. Thus, the multiset $W_{d'}(P)$ is uniquely determined by $SW_{d'}(P)$ and $UW_{d'}(P)$.

\end{proof}

\begin{lemma}
\label{lem:unique}
If $W_{d'}(P) = W_{d'}(Q)$, then $W_{d}(P) = W_{d}(Q)$, where $d|d'$.
\end{lemma}

\begin{proof}

Our proof strategy is that given $W_{d'}(P)$, we can reconstruct $W_{d}(P)$ \emph{uniquely}. If we can show this, then the proposition statement is justified.

Let $n = d'/d$. Label the counterclockwise oriented boundary minimal segments of $P$ as $\{ E^i \}_{i = 1} ^N$. Each $E^j$ is subdivided into $n$ $d'$-minimal segments $\{ E^{j,k} \}_{k = 1}^n$ when we regard $P$ as an $\mathcal{L}_{d'}$ polygon for the purposes of computing $W_{d'}(P)$. For fixed $j$, each segment $E^{j,k}$ has the same $d'$-weight since each lies in the same line (apply Proposition \ref{prop:lattice_distance} to see this). Therefore,

\begin{equation*}
W_{d'}(P) = \bigcup _{j = 1} ^N \left[ \cup _{k = 1} ^n \{ W_{d'}(E^{j,k}) \} \right].
\end{equation*}

Furthermore, we can compute the weight of $W_{d'}(E^{j,k})$ directly from $W_{d}(E^j)$. Suppose $E^j$ goes from $p = (w/d,x/d)$ to $q = (y/d, z/d)$. Then the edge from $p = (w/d, x/d)$ to $q' = p + \frac{1}{d'} (y - w, z - x)$ is in the set $\{ E^{j,k} \}_{j = 1} ^n$. WLOG, say this $d'$-minimal edge is $E^{j,1}$. Then,

\begin{align*}
W_{d'}(E^{j,1}) = \det \begin{pmatrix} nw & nw + (y - w) \\ nx & nx + (z - x) \end{pmatrix} = \det \begin{pmatrix} nw & y - w  \\ nx  & z - x \end{pmatrix} \\
= nwz - nwx - nxy + nxw = n(wz - xy) \\
= n \det \begin{pmatrix} w & y \\ x & z \end{pmatrix} \mod d'
\end{align*}.

We claim that there exists a unique choice $r$ of residue modulo $d$ such that

\begin{equation}
\label{eqn:compute1}
W_{d'}(E^{j,1}) = nr \mod d'
\end{equation}.

Equation \ref{eqn:compute1} says that for some integer $t$,

\begin{equation*}
W_{d'}(E^{j,1}) = nr + td' = nr + tnd = n(r + td).
\end{equation*}

Therefore, the residue $n$ divides $W_{d'}(E^{j,1})$ and we get

\begin{equation*}
(W_{d'}(E^{j,1})/n) = r + td \Leftrightarrow  r \equiv W_{d'}(E^{j,1})/n \mod d.
\end{equation*}

We conclude that

\begin{equation*}
W_{d}(P) = \bigcup _{j = 1} ^N \left \{ \left(W_{d'}(E^{j,1})/n\right) \right\}.
\end{equation*}

Indeed, $W_d(P)$ is uniquely determined by $W_{d'}(P)$.

\end{proof}

\begin{theorem}[Weight $W_d$ is an invariant for equidecomposability]
\label{thm:invariant}
Let $P$ be a denominator $d$ rational polygon. Let $\mathcal{F}: P \to Q$ be an equidecomposability relation. Then $W_d(P) = W_d(Q)$.
\end{theorem}

\begin{proof}
If $\mathcal{F}_{d'}: P \to Q$,\footnote{Recall by Remark \ref{rmk:denominator} that $d'$ must be divisble by $d$ if $\mathcal{F}_{d'}$ is an equidecomposability relation between denominator $d$ polygons $P$ and $Q$.} then by Lemmas \ref{lem:charge} and \ref{lem:current}, $SW_{d'}(P) = SW_{d'}(Q)$ and $UW_{d'}(P) = UW_{d'}(Q)$. By Lemmas \ref{lem:weight_agree} and Lemma \ref{lem:unique}, this implies $W_{d}(P) = W_{d}(Q)$, as desired.
\end{proof}

\begin{remark}
Note that all our definitions and results immediately generalize to the case where $P$ or $Q$ is a finite union of rational polygons. Moreover, nowhere have we used any assumptions about convexity, and in general, we make no assumptions about convexity for this entire paper.
\end{remark}

Theorem \ref{thm:invariant} comes with the following interesting corollary regarding minimal triangles.

\begin{corollary}
\label{cor:pseudoatomic}

If $S$ and $T$ are $d$-minimal triangles, then the following are equivalent.

\begin{enumerate}
\item $S$ and $T$ are $G$-equivalent.
\item $W(S) = W(T)$.
\item $S$ and $T$ are finitely rationally discretely equidecomposable.
\end{enumerate}

\end{corollary}

\begin{proof}

(1) $\Leftrightarrow$ (2) is a direct consequence of Theorem \ref{thm:weight_classification}. (3) $\Rightarrow$ (2) is the content of Theorem \ref{thm:invariant}. (1) $\Rightarrow$ (3) is true because a $G$-map from $S$ to $T$ is automatically an equidecomposability relation.

\end{proof}

In light of Corollary \ref{cor:pseudoatomic}, it is natural to try to classify the types of equidecomposability relations between $G$-equivalent minimal triangles $S$ and $T$. Are pure $G$-maps the only equidecomposability relations between $S$ and $T$? This question is related to a conjecture of Greenberg \cite{greenberg}. A positive answer would show, intuitively speaking, that minimal triangles behave like ``atoms'' with respect to equidecomposability relations. Such inquiries are revisited in Section \ref{sec:questions}.

\end{subsection}

\begin{subsection}{Ehrhart Equivalence Does not Imply Rational Finite Discrete Equidecomposability}

The computational software \texttt{LattE} \cite{LattE} can be used to show $\mathrm{ehr}_{T_{(1,2)}}(t) = \mathrm{ehr}_{T_{(1,4)}}(t)$.\footnote{The calculation can be done by hand. It suffices to check that $|tT_{(1,2)}\cap \mathbb{Z}^2|=|tT_{(1,4)}\cap \mathbb{Z}^2|$ for $1 \leq t \leq5$. } The explicit formulas are below.
    \begin{displaymath}
   \mathrm{ehr}_{T_{(1,2)}}(t) = \mathrm{ehr}_{T_{(1,4)}}(t) = \left\{
     \begin{array}{lr}
       \frac{x^2}{50}+\frac{3x}{50}-\frac{2}{25} & : t \equiv 1 \mod 5 \\
       \frac{x^2}{50}+\frac{x}{50}-\frac{3}{25} & : t \equiv 2 \mod 5 \\
       \frac{x^2}{50}-\frac{x}{50}-\frac{3}{25} & : t \equiv 3 \mod 5 \\
       \frac{x^2}{50}-\frac{3 x}{50}-\frac{2}{25} & : t \equiv 4 \mod 5 \\
       \frac{x^2}{50}+\frac{3x}{10} +1 & : t \equiv 5 \mod 5 \\
     \end{array}
   \right.
\end{displaymath}

However, by Theorem \ref{thm:invariant}, we see that $T_{(1,2)}$ and $T_{(1,4)}$ are not rationally, finitely equidecomposable because $W_5(T_{(1,2)}) = \{1, 1, -1 \} $ and $W_5(T_{(1,4)}) = \{ 2, -2, 1 \}$ This provides the partial negative answer to Haase--McAllister's \cite{haase} question of whether Ehrhart equivalence implies (general) equidecomposability (see Question \ref{que:ehr}).
\end{subsection}

\end{section}

\begin{section}{An Infinite Equidecomposability Relation}
\label{sec:infinite}
\begin{subsection}{Construction of the Infinite Equidecomposability Relation}
Previously, by section 3.3, we have shown that two special triangles $T_{(1,2)}$ and $T_{(1,4)}$ with the same Ehrhart quasi-polynomial are not equidecomposable. However, if we delete an edge from each triangle, there does exist an infinite equidecomposability relation mapping one to another. The existence of this infinite construction also explains why these two triangles share the same Ehrhart quasi-polynomial. It also raises many interesting problems discussed at the end of this section.

Label two edges of each of these triangles as follows. Let $e_1$ denote the closed edge of $T_{(1,2)}$ with endpoints $(\frac{1}{5},0)$, $(0,\frac{1}{5})$, let $e_2$ denote the closed edge with endpoints $(\frac{1}{5},0)$, $(\frac{1}{5},\frac{1}{5})$, let $e_3$ denote the closed edge of $T_{(1,4)}$ with endpoints $(\frac{2}{5},0)$, $(\frac{1}{5},\frac{1}{5})$, and let $e_4$ denote the closed edge with endpoints $(\frac{2}{5},0)$, $(\frac{2}{5},\frac{1}{5})$.

\begin{theorem}
\label{thm:infinite}
 Denote by $T_{(1,2)}$ the triangle with vertices $(\frac{1}{5},0)$, $(0,\frac{1}{5})$, $(\frac{1}{5},\frac{1}{5})$, and by $T_{(1,4)}$ the triangle with vertices $(\frac{2}{5},0)$, $(\frac{1}{5},\frac{1}{5})$, $(\frac{2}{5},\frac{1}{5})$. If we delete either one of $e_1$, $e_2$ from $T_{(1,2)}$ and either one of $e_3$, $e_4$ from $T_{(1,4)}$, the remaining $\Delta$-complexes are infinitely equidecomposable.
\end{theorem}
\begin{proof}

Since the unimodular transformation $U = \begin{pmatrix} 1 & 1 \\ 0 & 1 \end{pmatrix}$ maps $e_1$ to $e_2$ and $e_3$ to $e_4$, respectively, we delete either $e_1$ or $e_2$ from $T_{(1,2)}$ and either $e_3$ or $e_4$ from $T_{(1,4)}$, and claim that the remaining half-open $\Delta$-complexes \footnote{For our purposes, a $\Delta$-complex can be thought of as a disjoint union of open simplices. This is looser than the notion of simplicial complex because it is not required that the boundary of a face of a $\Delta$-complex be a part of the $\Delta$-complex, \emph{i.e.} half-open structures are allowed. See Chapter 2 of \cite{hatcher} for a more general topological definition of $\Delta$-complexes.} are infinitely equidecomposable.

Therefore WLOG we delete $e_1$ from $T_{(1,2)}$ and $e_3$ from $T_{(1,4)}$.
\\

\emph{Part I}: Choosing a $\Delta$-complex decomposition of $T_{(1,2)}$ and $T_{(1,4)}$ (\emph{Cutting})

Let $I = (\frac{1}{5},0)$. Construct all the lines ${\{l_{4,i}\}}$ connecting $I$ and the lattice points of $\frac{1}{5} \mathbb{Z} \times \frac{1}{5} \mathbb{Z}$ contained in the line $y = \frac{1}{5}$, starting from $(\frac{2}{5},\frac{1}{5})$ and going to the right. The $i^{th}$ line $l_{4,i}$ will divide $T_{(1,4)}$ into one more region. Denote the upper new region resulting from constructing $l_{4,i}$ as $R_i$. We restrict $R_i$ to be open. Let $r_{i,i+1}$ denote the open edge between $R_i$ and $R_{i+1}$. Also, $R_i$ has one side its open boundary bordering one of the two closed edges of $T_{(1,4)}$. We let  $n_i$ denote this open edge. Finally, let $N_i$ be the point of intersection of the line $l_{4,i}$ and the edge $e_4$. See Figure \ref{fig:infinite}. Thus we have
\begin{equation*}
T_{(1,4)}-e_4 = \bigsqcup_i \{R_i\} \bigsqcup_i \{r_{i,(i+1)}\} \bigsqcup_i \{n_i\} \bigsqcup_i \{N_i\}.
\end{equation*}

Next, choose another point $J = (\frac{2}{5},0)$. Construct all the lines ${\{l_{2,i}\}}$ connecting $J$ and the lattice points of $\frac{1}{5} \mathbb{Z} \times \frac{1}{5} \mathbb{Z}$ contained in the line $y = \frac{1}{5}$, starting from $(\frac{1}{5},\frac{1}{5})$ and going to the left. The $i^{th}$ line $l_{2,i}$ will divide $T_{(1,2)}$ into one more region. Denote the upper new open region resulted from cutting by $l_{2,i}$ as $S_i$. Let $s_{0,1}$ denote the open edge bordering $S_1$ that lies on the line $y = \frac{1}{5}$. Denote by $s_{i,i+1}$ the open edge lying between the regions $S_i$ and $S_{i+1}$. For each $S_i$, denote the open edge bordering $S_i$ that lies on $e_2$ of $T_{(1,2)}$ by $m_i$. Finally, let $M_i$ be the point of intersection of the line $l_{2,i}$ and the edge $e_2$. See Figure \ref{fig:infinite}. Thus,
\begin{equation*}
T_{(1,2)}-e_2 = \bigsqcup_i \{S_i\} \bigsqcup_i \{s_{(i-1),i}\} \bigsqcup_i \{m_i\} \bigsqcup_i \{M_i\}
\end{equation*}

\begin{figure}[H]
    \centering
    \includegraphics[keepaspectratio=true, width=15cm]{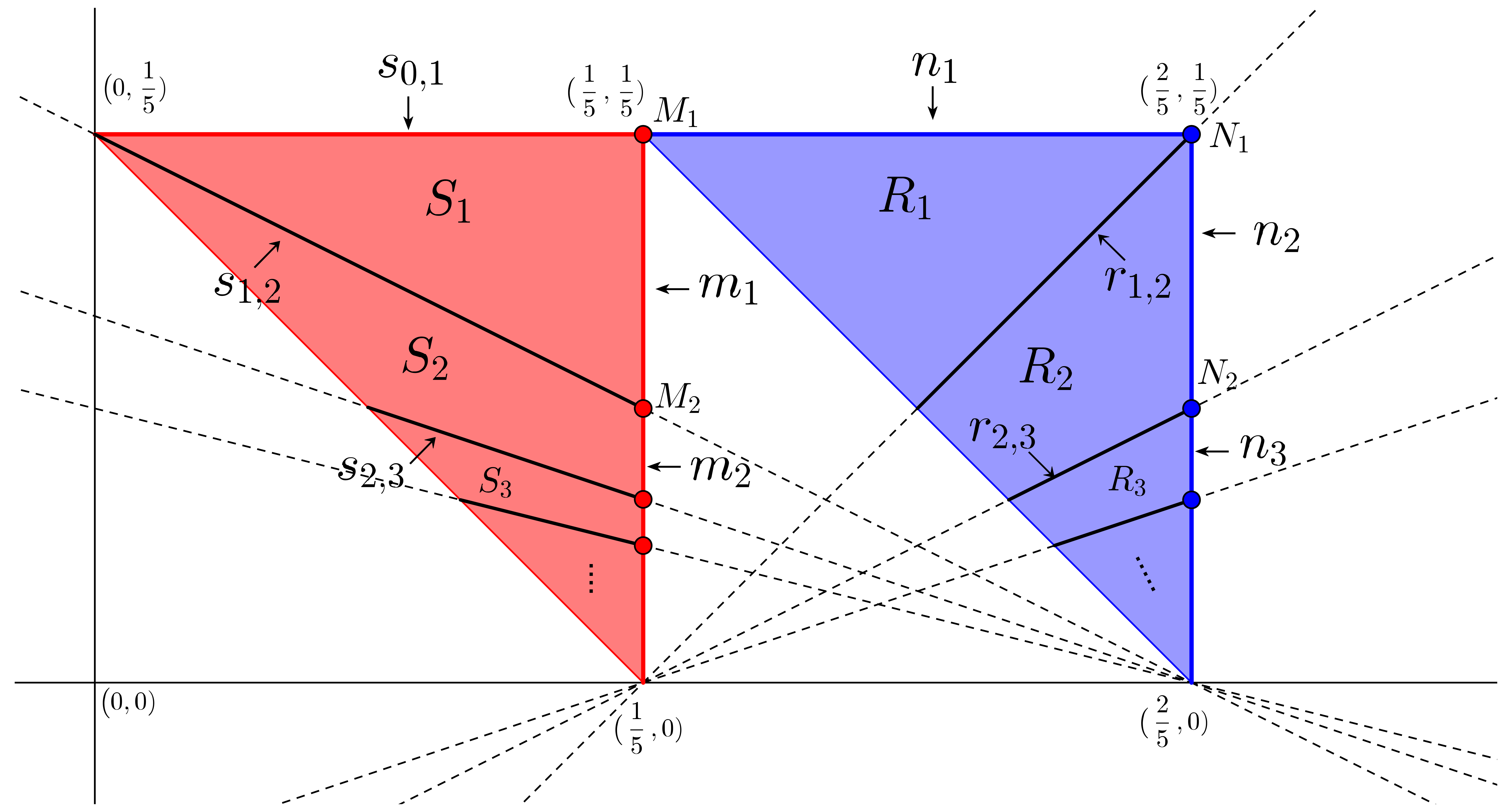}
    \caption{\small An Infinite Equidecomposability Relation. }
    \label{fig:infinite}
\end{figure}

Part II: Mapping the selected $\Delta$-complex decompositions (\emph{Pasting}).

Let $U_i = \begin{pmatrix} 1 & i \\ 0 & 1 \end{pmatrix}$. $U_i$ sends each $S_i$ to $R_i$, $s_{i-1,i}$ to $n_i$, $m_i$ to $r_{i,i+1}$, and $M_i$ to $N_i$. The details can be verified by writing down explicit expression of the coordinates of each piece. Intuitively, we can think of the transformation $U$ as mapping the triangle $T_{(1,2)}$ to cover $T_{(1,4)}$, step by step. Since $U_i$ fixes the points on $y=0$ but moves the lattice points on $y = \frac{1}{5}$ to the right by $i$ units of the lattice, $T_{(1,4)}$ is covered by $T_{(1,2)}$ completely.\\

\end{proof}

It is natural to ask if the Ehrhart quasi-polynomial is even preserved in general by equidecomposability relations consisting of an infinite amount of simplices. To do so, let's first provide the general definition of an Ehrhart function of a subset of $\mathbb{R}^2$.

Here we denote as $\mathrm{ehr}(t)$, the Ehrhart function (also referred to as the Ehrhart counting function), the number of lattice points in the $t^{th}$ dilate of a general subset $S \subset \mathbb{R}^2$. This function may not have a exact form, as in the case of rational polygons. However, whether or not it is a quasi-polynomial, the Ehrhart function is still well-defined. Moreover, observe that the definition provided below agrees with the Ehrhart function of a polygon provided in Equation \ref{eqn:ehrhart} of Section \ref{sec:introduction}.

\begin{definition}[Ehrhart function of a subset of $\mathbb{R}^2$]

Let $S$ be a bounded subset of $\mathbb{R}^2$. Then the \emph{Ehrhart function} of $S$ is defined to be

\begin{equation*}
\mathrm{ehr}_S(t) := |tS \cap \mathbb{Z}^2|
\end{equation*}

where $tS$ denotes the $t^{th}$ dilate of $S$.

\end{definition}

We require $S$ to be bounded so that $\mathrm{ehr}_S(t)$ is always finite. The fact that Ehrhart functions are preserved by potentially infinite equidecomposability relations follows fairly quickly from the definitions.

\begin{proposition}
\label{prop:ehr_fn}
Let $S$ and $S'$ be bounded subsets of $\mathbb{R}^2$ with (not necessarily finite or rational) $\Delta$-complex decompositions $\mathcal{T}$ and $\mathcal{T}'$, respectively. If $\mathcal{F}: (S, \mathcal{T}) \to (S', \mathcal{T}')$ is an equidecomposability relation, then

\begin{equation*}
\mathrm{ehr}_S(t) = \mathrm{ehr}_{S'}(t).
\end{equation*}

\end{proposition}

\begin{proof}

Let $t$ be a positive integer. We want to construct a bijection $\Phi$ from $\mathbb{Z}^2 \cap tS$ to $\mathbb{Z}^2 \cap tS'$. Given $p \in \mathbb{Z}^2 \cap tS$, define $\Phi(p) = t \mathcal{F}(p/t)$. Injectivity of $\Phi$ is clear because both the scaling map $p \to p/t$, equidecomposability relation, and dilation map $\mathcal{F}(p/t) \to t \mathcal{F}(p/t)$ are all injective. Also, $t \mathcal{F}(p/t)$ is an integer point because $p$ is an integer point and $\mathcal{F}$ preserves denominators.

Now for surjectivity. Suppose $tq \in tS' \cap \mathbb{Z}^2$. Then $q \in S'$. Since $\mathcal{F}$ is surjective, there exists $p \in S$ such that $\mathcal{F}(p) = q$. Since $tp \in tP$, we observe that $\Phi(tp) = t \mathcal{F}(p) = tq$. Again, since $\mathcal{F}$ preserves denominators, if $tq$ is an integer point, it follows that $tp$ is an integer point, as desired.

Indeed, $\mathbb{Z}^2 \cap tS$ is in bijection with $\mathbb{Z}^2 \cap tS'$. Moreover, both sets are finite by the boundedness of $S$ and $S'$.

\end{proof}

We observe that the construction from Theorem \ref{thm:infinite} can be applied to a much larger family of pairs of triangles.

\begin{definition}
Let $T_l$ and $T_r$ be two minimal triangles with denominator $d$. The pair of triangles $T_l$ and $T_r$ are called \emph{similar neighbors} if $T_l$ is G-equivalent to $T_l'$ and $T_r$ is G-equivalent to $T_r'$, where $T_l' = \mathrm{Conv} \left((\frac{1+t}{d},\frac{1}{d}), (\frac{1+t}{d}, 0), (\frac{t}{d}, \frac{1}{d}) \right)$, $T_r' = \mathrm{Conv} \left((\frac{2+t}{d},\frac{1}{d}), (\frac{2+t}{d}, 0), (\frac{1+t}{d}, \frac{1}{d}) \right)$ for some fixed integer t, where $1+t$, $2+t$ and $d$ are pairwise relatively prime, and $3+2t \neq 0$ mod $d$.\footnote{When $3+2t = 0$ mod $d$, $T_l$ and $T_r$ are actually G-equivalent, so we exclude this case.}

\end{definition}

\begin{remark}
Observe that when $d<5$ there are no such similar neighbors. One example of similar neighbors, the pair of our favorite two triangles $T_{(1,2)}$ and $T_{(1,4)}$ is the similar neighbors with the smallest denominator.
\end{remark}
The following figure \ref{fig:sim_neighbors} illustrates the case when $d=6$.

\begin{figure}[H]
    \centering
    \includegraphics[keepaspectratio=true, width=8cm]{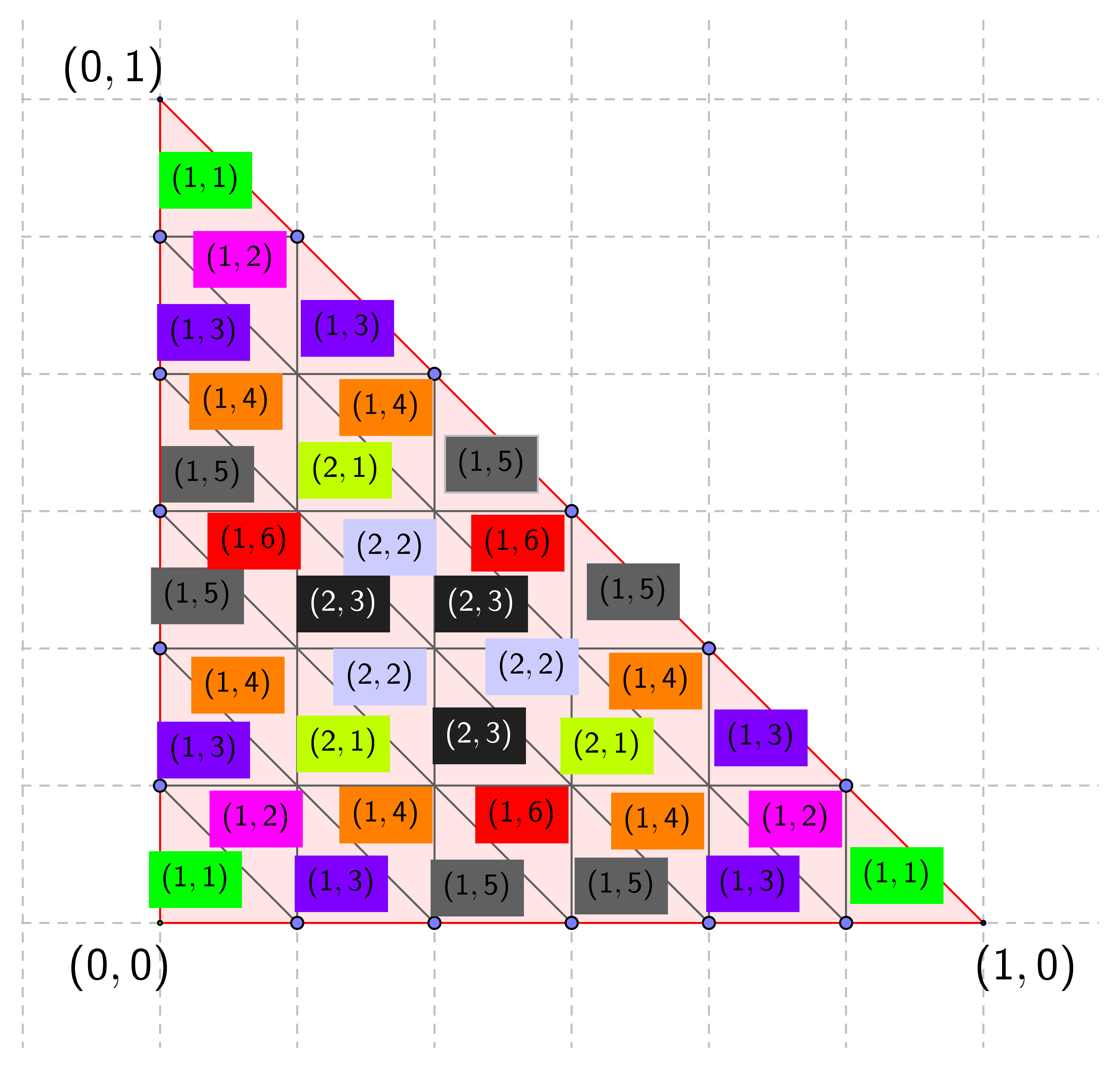}
    \caption{\small  $T_{(1,2)}$, $T_{(1,4)}$ and $T_{(1,6)}$ don't have a 6-primitive vertices on x-axis, thus they don't form similar neighbors. There does not exist a pair of similar neighbors when when $d=6$. }
    \label{fig:sim_neighbors}
\end{figure}

\begin{theorem}
If two triangles $T_l$ and $T_r$ are similar neighbors, then they share the same Ehrhart quasi-polynomials, but are not equidecomposable.
\end{theorem}
\begin{proof}
WLOG, let $T_l = \mathrm{Conv} \left((\frac{1+t}{d},\frac{1}{d}), (\frac{1+t}{d}, 0), (\frac{t}{d}, \frac{1}{d}) \right)$, $T_r = \mathrm{Conv} \left((\frac{2+t}{d},\frac{1}{d}), (\frac{2+t}{d}, 0), (\frac{1+t}{d}, \frac{1}{d}) \right)$. First let's show they share the same Ehrhart quasi-polynomials. WLOG delete the edge $e_l$ with endpoints $(\frac{1+t}{d}, 0), (\frac{1+t}{d}, \frac{1}{d})$ from $T_l$ and edge $e_r$ with endpoints $(\frac{2+t}{d}, 0), (\frac{2+t}{d}, \frac{1}{d})$ from $T_r$. By the same construction in \ref{thm:infinite}, we can map the $\Delta$-complex $(T_l-e_l)$ to $(T_r-e_s)$. By Proposition \ref{prop:ehr_fn}, $(T_l-e_l)$ and $(T_r-e_s)$ share the same Ehrhart quasi-polynomial. It remains to show the edge $e_l$ and $e_r$ have the same Ehrhart quasi-polynomial.  Because $1+t$, $2+t$ and $d$ are relatively prime, we have a bijection between the denominator $d'$ points in $e_l$ and the denominator $d'$ points\footnote{By \textit{denominator $d'$-point}, we mean a point that lies in the lattice $\mathcal{L}_{d'} = \frac{1}{d'} \mathbb{Z} \times \frac{1}{d'} \mathbb{Z}$.} in $e_r$ by projecting the segment $e_l$ onto $e_r$. Thus $P$ and $Q$ share the same Ehrhart quasi-polynomial.

To show they are not equidecomposable, we compute $W_d(T_l) = \{1+t,1,-1-t \}$ and $W_d(T_{(1,2)}) = \{2+t, 1, -2-t \}$. Since $3+2t \neq 0$ mod $d$, $W_d(T_l) \neq W_d(T_r)$. This completes the proof.

\end{proof}

\end{subsection}

\begin{subsection}{Intractable and Tractable Equidecomposability Relations}

Referring back to Theorem \ref{thm:infinite}, the question remains as to why we must delete an edge from each triangle. Since we allow decompositions consisting of infinitely many pieces, can we also find an infinite equidecomposable relation between $e_1$ and $e_3$? This depends on the restrictions on decompositions, leading us to consider special types of decompositions.

\begin{conjecture}
\label{conj:infinite}
Let $S$ and $S'$ be bounded subsets of $\mathbb{R}^2$ with $\Delta$-complex decompositions $\mathcal{T}$ and $\mathcal{T}'$, respectively. There exists an equidecomposability relation $\mathcal{F}: (S, \mathcal{T}) \to (S', \mathcal{T}')$ if only if

\begin{equation*}
\mathrm{ehr}_S(t) = \mathrm{ehr}_{S'}(t).
\end{equation*}
\end{conjecture}

\begin{remark}
We conjecture that the Ehrhart function is a necessary and sufficient criterion for equidecomposability if we allow arbitrary decompositions. We proved the forward direction. The key difficulty of the backward direction is that it is hard to keep track of the irrational points in an arbitrary decomposition, although their Ehrhart function is trivially 0.
\end{remark}

This raises a question, what kind of decomposition should be considered?  We would prefer a more restricted decomposition that could be constructed explicitly.

\begin{definition}
Let $S$ and $S'$ be bounded subsets of $\mathbb{R}^2$ with $\Delta$-complex decompositions $\mathcal{T}$ and $\mathcal{T'}$, respectively, consisting entirely of 0-simplices (e.g., $S=\bigsqcup S_i$ where $S_i$ is a vertex), and $\mathcal{F}: (S, \mathcal{T}) \to (S', \mathcal{T}')$ is an equidecomposability relation. Recall that $\mathcal{F}$ may be regarded as an assignment of $G$-maps $g_i \in G$ to the points $S_i$ in $\mathcal{T}$. We say that $g_i \in \mathcal{F}$ if $g_i$ is  assigned to some face in $\mathcal{T}$. Then we define a \emph{maximal subcomplex} $M(g_i)$ for every $g_i \in \mathcal{F}$ to be the set consisting of 0-simplices as follows,
    \begin{equation*}
    M(g_i) := \{S_j|g_i\ acts\ on\ S_j\ sending\ to\ S_j'\ ,\ S_j \in \mathcal{T}\ S_j' \in \mathcal{T'}\}
    \end{equation*}
    Furthermore, we define a \emph{maximal decomposition} of $\mathcal{F}$ as follows,
    \begin{equation*}
    \mathcal{T}_M(\mathcal{F}) := \{M(g_i)|g_i\in\ \mathcal{F}\}.
    \end{equation*}
\end{definition}

To be clear, $M(g_i)$ is the union of all vertices assigned by the same $g_i \in \mathcal{F}$, and it is in general uncountable.

\begin{remark}
Observe that even though the cardinality of $\mathcal{T}$ is uncountable in general, the number of distinct $G$-maps occuring in $\mathcal{F}$ is a countable set at most, because the group $G$ is countable. Therefore $\mathcal{T}_M(\mathcal{F})$ is also countable. This implies some of the maximal subcomplexes contain uncountably many 0-simplices. We only need to worry about decompositions formed by restricting to the structures of these maximal subcomplexes.
\end{remark}

\begin{definition}

Let $S$ and $S'$ be bounded subsets of $\mathbb{R}^2$ with some $\Delta$-complex decompositions $\mathcal{T}$ and $\mathcal{T'}$, respectively, and suppose there exists an equidecomposability relation $\mathcal{F}: (S, \mathcal{T}) \to (S', \mathcal{T}')$. If every maximal subcomplex $M(g_i)$ induced by this equidecomposability relation is a disconnected 0-simplex space, and $S$ and $S'$ are not a finite set of points, then this equidecomposability relation is called an \emph{intractable} equidecomposability relation. Otherwise, it is referred to as \emph{tractable}.
\end{definition}

\begin{remark}
Here we make a broader definition of equidecomposability relation where we allow a maximal decomposition $\mathcal{T}_M(\mathcal{F})$ to be a countable set instead of a finite set as in Definition \ref{def:equi}.
\end{remark}

However, the Ehrhart function turns out to not be a necessary and sufficient criterion even for a tractable equidecomposability relation. It can easily be computed by hand that $\mathrm{ehr}_{e_1}(t) = \mathrm{ehr}_{e_3}(t)$. However, the following shows that the two segments are not tractably equidecomposable.

\begin{proposition}{There does not exist a tractable equidecomposability between edges $e_1$ with endpoints $(\frac{1}{5},0)$, $(0,\frac{1}{5})$ and $e_3$ with endpoints $(\frac{2}{5},0)$, $(\frac{1}{5},\frac{1}{5})$}
\end{proposition}

\begin{proof} Suppose there exists a tractable equidecomposability relation $\mathcal{F}: (e_1, \mathcal{T}_1) \to (e_3, \mathcal{T}_3)$ for some $\Delta$-complex decompositions $\mathcal{T}_1$ and $\mathcal{T}_3$ of $e_1$ and $e_3$, respectively. Then $\mathcal{T}_1$ must contain a maximal subcomplex whose union contains an open segment $e'$, because $e_1$ consists of an uncountable number of points. By the discussion preceding Remark \ref{rmk:denominator}, $e'$ is assigned a $G$-map $g$ with the property

\begin{equation*}
\mathcal{F}(e') = g(e').
\end{equation*}

Moreover, $g(e')$ is an open segment in $\mathcal{T}_3$. Recall that $G$-maps send $5$-minimal segments to $5$-minimal segments. Therefore, since $e' \subset e_1$, $g(e_1)$ is the $5$-minimal segment containing $g(e')$. However, there can be at most one $5$-minimal segment containing a given open segment. It follows that $g(e_1) = e_3$. Yet this is a contradiction because $\pm 1 = W(e_1) \neq W(e_3) = \pm 2 \mod 5$ and Proposition \ref{prop:weight_inv} says that the weight of a minimal edge is preserved up to sign by $G$-maps.

\end{proof}

We recall Conjecture \ref{conj:infinite}. Now to be more precise, we conjecture the Ehrhart function is a necessary and sufficient criterion for equidecomposability if we allow intractable equidecomposability relations. An example of a possible maximal subcomplex in an intractable equidecomposability relation is Cantor set, which is a disconnected space of points that has uncountable cardinality. We also wonder if there exists a tractable equidecomposability relation between a pair of similar neighbors.
\end{subsection}
\end{section}

\begin{section}{Further Questions}
\label{sec:questions}

We briefly summarize some questions and directions for further inquiry, restricting to the case of polygons in accordance with the style of this paper.

\begin{enumerate}

\item At the end of Section \ref{sec:weight}, we observed that $d$-minimal triangles are equidecomposable if and only if they are $G$-equivalent. What other polygons or regions have this type of property? Moreover, are $G$-equivalences the only possible type of equidecomposability relations between the interiors of $d$-minimal triangles? The former question is part of a conjecture of Greenberg \cite{greenberg}. A positive answer would say, intuitively speaking, that $d$-minimal triangles behave like ``atoms''.

\smallskip

\item What invariants can be found for equidecomposability of polygons if we allow irrational simplices? What if we allow an infinite amount of simplices? Is there some way to generalize our progress to these cases? Does there exist a tractable equidecomposablility relation between $T_{(1,2)}$ and $T_{(1,4)}$? Questions about infinite equidecomposability relations were visited in Section \ref{sec:infinite}. In particular, see Conjecture \ref{conj:infinite}.

\end{enumerate}

\end{section}

\begin{section}{Acknowledgments}
This research was conducted during Summer@ICERM 2014 at Brown University and was generously supported by a grant from the NSF. First and foremost, the authors express their deepest gratitude to Sinai Robins for introducing us to the problem, meeting with us throughout the summer to discuss the material in great detail, suggesting various useful approaches, answering our many questions, and critically evaluating our findings. In the same vein, we warmly thank our research group's TA's from the summer, Tarik Aougab and Sanya Pushkar, for numerous beneficial conversations, suggestions, and verifying our proofs. In addition, we thank the other TA's, Quang Nhat Le and Emmanuel Tsukerman, for helpful discussions and ideas. We also thank Tyrrell McAllister for visiting ICERM, providing an inspiring week-long lecture series on Ehrhart theory, and discussing with us our problems and several useful papers in the area. We are grateful to Jim Propp for a very long, productive afternoon spent discussing numerous approaches and potential invariants for discrete equidecomposability. We also thank Hugh Thomas, who is a professor at the second author's home university, for interesting discussions as well as carefully reviewing and commenting on drafts of these results. We extend our gratitude to Michael Mossinghoff and again to Sinai Robins for coordinating the REU. Finally, we thank the ICERM directors, faculty, and staff for providing an unparalleled research atmosphere with a lovely view.
\end{section}

\bibliographystyle{alphaurl}
\bibliography{discrete_equi_of_polygons}

\end{document}